\documentclass[12pt,a4paper,reqno]{amsart}

\usepackage{amsfonts}
\usepackage{amssymb}
\usepackage{amsthm}
\usepackage{amsmath}
\usepackage{a4wide}
\usepackage{mathrsfs}
\usepackage{epsfig}
\usepackage{tikz,fp,ifthen}
\usepackage{float}
\usepackage{graphicx}
\usepackage{enumitem}
\usepackage{esint}
\usepackage{verbatim} 
\usepackage{palatino}
\usepackage{anyfontsize}

\usepackage{mathtools}
\mathtoolsset{showonlyrefs}

\usepackage{color}
\definecolor{citegreen}{rgb}{0,0.6,0}
\definecolor{refred}{rgb}{0.8,0,0}
\usepackage[colorlinks,urlcolor=blue,citecolor=citegreen,linkcolor=refred]{hyperref}
\usepackage{pifont}

\theoremstyle{plain}

\newtheorem{thm}{Theorem}[section]
\newtheorem{lem}[thm]{Lemma}
\newtheorem{proposition}[thm]{Proposition}

\theoremstyle{remark}
\newtheorem{rem}[thm]{Remark}
\newtheorem{remark}[thm]{Remark}

\theoremstyle{definition}
\newtheorem{definition}[thm]{Definition}

\numberwithin{equation}{section} 

\def\OO{{\mathrm{O}}}
\def\II{{\mathrm{I}}}
\newcommand{\T}{\mathbb{T}}
\newcommand{\R}{\mathbb{R}}
\newcommand{\N}{\mathbb{N}}
\def\Z{\mathbb Z}
\newcommand{\SSS}{\mathbb S}
\newcommand{\RRR}{{\mathrm R}}

\newcommand{\Riem}{{\mathrm {Riem}}}

\newcommand{\Lebes}{\mathscr{L}}

\DeclarePairedDelimiter{\norma}{\lVert}{\rVert}
\DeclarePairedDelimiter{\abs}{\lvert}{\rvert}
\newcommand{\E}{\mathcal E}
\newcommand{\dmu}{d \mu}

\newcommand{\pa}{\partial}

\newcommand{\HHH}{\mathrm{H}} 
\newcommand{\BBB}{\mathrm{B}} 
  
\newcommand{\vol}{\mathrm{Vol}}
\newcommand{\grad}{\nabla}
\newcommand{\eps}{\varepsilon}
\newcommand{\medint}{-\kern -,375cm\int}
\newcommand{\medintinrigo}{-\kern -,315cm\int}

\newcommand{\Htilde}{\widetilde{H}}

\newcommand{\beq}{\begin{equation}}
\newcommand{\eeq}{\end{equation}}
\def\Div{\operatorname*{div}\nolimits}

\newcommand{\A}{\mathcal A}
\def\Ri{\R^{n}}

\def\bbigstar{\operatornamewithlimits{\text{\Large{$\circledast$}}}}

\def\pol{{\mathfrak{p}}}
\def\qol{{\mathfrak{q}}}

\newcommand{\Tort}{T^{\perp}}

\DeclareRobustCommand{\rchi}{{\mathpalette\irchi\relax}}
\newcommand{\irchi}[2]{\raisebox{\depth}{$#1\chi$}}

\title{Stability for the Surface Diffusion Flow}

\author[A.~Diana]{Antonia Diana}
\address{A.~Diana, Scuola Superiore Meridionale, Italy}
\email{antonia.diana@unina.it}

\author[N.~Fusco]{Nicola Fusco}
\address{N.~Fusco, Universit\`a di Napoli Federico II \& Scuola Superiore Meridionale, Italy}
\email{n.fusco@unina.it}

\author[C.~Mantegazza]{Carlo Mantegazza}
\address{C.~Mantegazza, Universit\`a di Napoli Federico II \& Scuola Superiore Meridionale, Italy}
\email{carlo.mantegazza@unina.it}

\date{\today}

\begin{document}

\begin{abstract} We study the global existence and stability of surface diffusion flow (the normal velocity is given by the Laplacian of the mean curvature) of smooth boundaries of subsets of the $n$--dimensional flat torus. More precisely, we show that if a smooth set is ``close enough'' to a strictly stable critical set for the Area functional under a volume constraint, then the surface diffusion flow of its boundary hypersurface exists for all time and asymptotically converges to the boundary of a ``translated'' of the critical set. This result was obtained in dimension $n=3$ by Acerbi, Fusco, Julin and Morini in~\cite{AcFuMoJu} (extending previous results for spheres of Escher, Mayer and Simonett~\cite{escmaysim}, Wheeler~\cite{Wheeler,Wheeler:2013aa} and Elliott and Garcke~\cite{EllGar}). Our work generalizes such conclusion to any dimension $n\in\N$. For sake of clarity, we show all the details in dimension $n=4$ and we list the necessary modifications to the quantities involved in the proof in the general $n$--dimensional case, in the last section.
\end{abstract}

\maketitle

\tableofcontents

\section{Introduction}

Given a smooth immersed hypersurface in an $n$--dimensional flat torus $\varphi=\varphi_0:M\to \T^n$ (or in $\Ri$), we say that a smooth family of smooth embeddings $\varphi_t:M\to \T^n$, for $t \in[0,T)$, is a {\em surface diffusion flow} for $\varphi_0$ if 
\begin{equation}\label{sdf2}
\frac{\partial\varphi_t}{\partial t}=(\Delta\HHH)\nu\,,
\end{equation}
that is, the outer normal velocity (here $\nu$ is the outer normal) of the moving hypersurfaces is given by the Laplacian (in the induced metric) $\Delta\HHH$ of the mean curvature, at every point and time. Such flow was first proposed by Mullins in~\cite{Mullins} to study thermal grooving in material sciences and first analyzed mathematically in more detail in~\cite{escmaysim}. In particular, in the physically relevant case of three--dimensional space, it describes the evolution of interfaces between solid phases of a system, driven by surface diffusion of atoms under the action of a chemical potential (see for instance~\cite{GurJab}).

Studying the flow in a flat torus $\T^n$, described as the quotient of $\R^n$ by a discrete group of translations generated by some $n$ linearly independent vectors, is equivalent to consider the flow of ``periodic'' hypersurfaces, invariant by such group of translations. Then, it is clear that our analysis also applies to compact hypersurfaces in $\R^n$ or, more in general, in any (generalized) ``cylinder'' $\SSS^1\times\dots\times\SSS^1\times\R\times\dots\times\R$ of dimension $n$, with a flat metric.

Notice that, by the general equality $\Delta\varphi=-\HHH\nu$ (see equation~\eqref{lap} below), the system~\eqref{sdf2} can be rewritten as
\begin{equation}\label{paraeq}
\frac{\partial\varphi_t}{\partial t}=-\Delta\Delta\varphi_t+\text{ lower order terms}
\end{equation}
hence, it is a fourth order, {\em quasilinear} and {\em degenerate}, parabolic system of PDEs. Indeed, it is quasilinear, as the coefficients (as second order partial differential operator) of the Laplacian associated to the induced metrics on the evolving hypersurfaces, depend on the first order derivatives of $\varphi_t$ (and the coefficient of $\Delta\Delta$ on the third order derivatives) and the operator at the right hand side of system~\eqref{sdf2} is degenerate, as its symbol (the symbol of the linearized operator) admits zero eigenvalues, due to the invariance of the Laplacian by diffeomorphisms. The lack of maximum principle, as the flow is of fourth order, implies that it does not preserve convexity (see~\cite{Ito}), nor the embeddedness (see~\cite{gigaito1}), indeed it also does not have a ``comparison principle'', while it is invariant by isometries of $\T^n$, reparametrizations and tangential perturbations of the velocity of the motion. Moreover, when it is restricted to closed embedded hypersurfaces which are boundaries of sets, the enclosed volume is preserved and actually it can be regarded as the $\widetilde{H}^{-1}$--gradient flow of the {\em volume--constrained Area functional} (see~\cite{DDMsurvey} or~\cite{Gar}, for instance). 

There holds the following short time existence and uniqueness theorem (and also of dependence on the initial data) for the surface diffusion flow starting from a smooth hypersurface, proved by Escher, Mayer and Simonett in~\cite{escmaysim}, which should be expected by the parabolic nature of the system~\eqref{sdf2}, as shown by formula~\eqref{paraeq}. The original result deals with the evolution in the whole space $\R^n$ of a generic hypersurface, even only immersed, hence possibly with self--intersections. It is anyway straightforward to adapt the same arguments to our case, when the ambient is a flat torus $\T^n$ (and the hypersurfaces are boundaries of sets).

\begin{thm}\label{th:EMS0}
Let $\varphi_0:M\to\R^n$ be a smooth and compact, immersed hypersurface. Then, there exists a unique smooth $\varphi:[0,T)\times M\to\R^n$ such that $\varphi_t=\varphi(t,\cdot)$ is the surface diffusion flow of $\varphi_0$, that is a solution of equation~\eqref{sdf2}, for some maximal time of existence $T>0$. Moreover, such flow and the maximal time of existence depend continuously on the $C^{2,\alpha}$--norm of the initial hypersurface $\varphi_0$.
\end{thm}

Actually, it is very likely, even if an explicit example is not present in literature (up to our knowledge), that this flow could develop singularities in finite time (as the mean curvature flow). Anyway, in the same mentioned paper~\cite{escmaysim} by Escher, Mayer and Simonett, the authors also showed that if the initial hypersurface is $C^{2,\alpha}$--close enough to a sphere with the same enclosed volume, then the flow exists for every time and converges smoothly to a translate of such sphere. The analogous result was obtained by Wheeler in~\cite{Wheeler}, for surfaces and in~\cite{Wheeler:2013aa}, for closed plane curves (see also the work of Elliott and Garcke~\cite{EllGar} for curves) with a weaker initial $W^{2,2}$--closedness condition and by Escher and Mucha in a previous work~\cite{eschmuch} with a Besov--type condition. Furthermore, in~\cite{wheeler22} Wheeler showed that any surface diffusion flow of curves that exists for all time, must converge smoothly, exponentially fast to a multiply--covered circle. We also mention a work by Miura e Okabe~\cite{miuraokabe19} where the authors proved a global existence result provided that an initial curve is $W^{2,2}$--close to a multiply covered circle and sufficiently rotationally symmetric. Later on, Acerbi, Fusco, Julin and Morini in~\cite{AcFuMoJu} extended these results, in dimensions two and three, to hypersurfaces close to boundaries of {\em strictly stable critical sets} (that we are going to define in a while) for the volume constrained Area functional (as it is every ball). Our aim in this work is to generalize such stability conclusion to any dimension $n\in\N$. Because of several heavy analytic and algebraic computations needed in the analysis, we present in full detail the proof in dimension $n=4$ and we list the appropriate modifications for the general $n$–dimensional case in the last section. The choice of $n=4$ is not merely pragmatic: this dimension already exhibits the full range of analytic and algebraic difficulties arising in higher dimensions, while still allowing for a complete and reasonably transparent presentation. Carrying out all the computations in general dimension would have made the exposition unnecessarily technically heavy. For this reason, we restrict to $n=4$ as a representative case, while the explicit and detailed computations and proofs in full generality can be found in the PhD thesis of Antonia Diana~\cite{DianaPhD}.
\section{Preliminaries}

We introduce the basic notations and facts about hypersurfaces that we need in the paper, possible references are~\cite{gahula} or the first part of~\cite{petersen1}.

We will consider closed smooth hypersurfaces in the $n$--dimensional torus $\T^n\!\approx\R^n/\Z^n$ or in $\R^n$, given by smooth immersions $\varphi:M \to \T^n$ of a smooth, $(n-1)$--dimensional, compact manifold $M$, representing a hypersurface $\varphi(M)$ of $\T^n$. Taking local coordinates around any $p\in M$, we have local bases of the tangent space $T_p M$, which can be identified with the $(n-1)$--dimensional hyperplane $d\varphi_p(T_pM)$ of $\R^n\approx T_{\varphi(p)}\T^n$ which is tangent to $\varphi(M)$ at $\varphi(p)$ and of the cotangent space $T_p^{*}M$, respectively given by vectors $\bigl\{\frac{\partial\,}{\partial x_i}\bigr\}$ and 1--forms $\{dx_j\}$. So, we will denote vectors on $M$ by $X=X^i$, which means 
$X=X^i\frac{\partial\,}{\partial x_i}$, covectors by
$Y=Y_j$, that is, $Y=Y_jdx_j$ and a general mixed tensor with
$T=T^{i_1\dots i_k}_{j_1\dots j_l}$.

\smallskip

{\em In the whole paper the convention to sum over repeated 
indices will be adopted.}

\smallskip

Sometimes we will need also to consider tensors along
$M$, viewing it as a submanifold of $\T^n$ or $\Ri$ via the map $\varphi$, in
that case we will use the Greek indices to denote the components of
such tensors in the canonical basis $\{e_\alpha\}$ of $\Ri$, for
instance, given a vector field $X$ along $M$, not necessarily tangent,
we will have $X=X^\alpha e_\alpha$.

The manifold $M$ gets in a natural way a metric tensor $g$, pull--back via the map $\varphi$ of the metric tensor of $\T^n$, coming from the standard scalar product $\langle\cdot\,|\,\cdot\rangle$ of $\R^n$, hence, turning it into a Riemannian manifold $(M,g)$. Then, the components of $g$ in a local chart are
$$
g_{ij}=\left \langle\frac{\pa \varphi}{\pa x_i}\,\Big \vert \frac{\pa \varphi}{\pa x_j}\right \rangle
$$
and the ``canonical'' measure $\mu$, induced on $M$ by the metric $g$ is then locally described by $\mu=\sqrt{\det g_{ij}}\,{\Lebes}^{n-1}$, where ${\Lebes}^{n-1}$ is the standard Lebesgue measure on $\R^{n-1}$.\\
The inner product on $M$, extended to tensors, is given by
$$
g(T , S)=g_{i_1s_1}\dots g_{i_k s_k}g^{j_1
z_1}\dots g^{j_l z_l} T^{i_1\dots
 i_k}_{j_1\dots j_l}S^{s_1\dots
 s_k}_{z_1\dots z_l}
$$
where $g_{ij}$ is the matrix of the coefficients of the metric
tensor in the local coordinates and $g^{ij}$ is its 
inverse. Clearly, the norm of a tensor is then
$$
\vert T\vert=\sqrt{g(T , T)}\,.
$$

The induced Levi--Civita covariant derivative on $(M,g)$ of a vector field $X$ and of a 1--form $\omega$ 
are respectively given by
$$
\nabla _jX^i=\frac{\partial X^i}{\partial
 x_j}+\Gamma^{i}_{jk}X^k\,, \qquad \nabla _j\omega_i=\frac{\partial \omega_i}{\partial
 x_j}-\Gamma^k_{ji}\omega_k\,, 
$$
where $\Gamma^{i}_{jk}$ are the Christoffel symbols of the connection $\nabla$, expressed
by the formula
$$
\Gamma^{i}_{jk}=\frac{1}{2} g^{il}\Bigl(\frac{\partial\,}{\partial
 x_j}g_{kl}+\frac{\partial\,}{\partial
 x_k}g_{jl}-\frac{\partial\,}{\partial
 x_l}g_{jk}\Bigr)\,.
$$
The covariant derivative $\nabla T$ of a tensor
$T=T^{i_1\dots i_k}_{j_1\dots j_l}$ will be denoted by 
$\nabla_sT^{i_1\dots i_k}_{j_1\dots j_l}=(\nabla T)^{i_1\dots
 i_k}_{sj_1\dots j_l}$ and with $\nabla^m T$ we will mean the $m$--th iterated covariant
derivative of a tensor $T$.

The gradient $\nabla f$ of a function, the divergence $\Div X$ of a tangent vector field and the Laplacian $\Delta f$ at a point $p \in M$, are defined respectively by
$$
g(\nabla f(p) , v)=df_p(v)\qquad\forall v\in T_p M\,,
$$
\begin{equation}\label{divformcar}
\Div X={\mathrm {tr}} \nabla X=\nabla _iX^i=\frac{\partial X^i}{\partial x_i}+\Gamma^{i}_{ik}X^k
\end{equation}
(in a local chart) and $\Delta f =\Div\nabla f$. The Laplacian $\Delta T$ of a tensor $T$ is $\Delta T=g^{ij}\nabla_i\nabla_jT$.
We then recall that by the {\em divergence theorem} for compact manifolds (without boundary), there holds
\begin{equation}\label{divteo}
\int_{M}\Div X\,d\mu=0\,,
\end{equation}
for every tangent vector field $X$ on $M$, which in particular implies
\begin{equation}\label{corollariodivteo}
\int_{M}\Delta f\,d\mu=0\,,
\end{equation}
for every smooth function $f:M \to\R$.

Assuming that we have a globally defined unit {\em normal} vector field $\nu:M\to\R^n$ to $\varphi(M)$ (this will hold in our situation where the hypersurfaces are boundaries of sets $E\subseteq\T^n$, hence we will always consider $\nu$ to be the {\em outer unit normal vector} at every point of $\pa E$), we define the {\em second fundamental form} $\BBB$ which is a symmetric bilinear form given, in a local charts, by its components
\begin{equation}\label{secform}
h_{ij} = - \biggl \langle \frac{\pa^2 \varphi}{\pa x_i \pa x_j}\,\biggr\vert \,\nu \biggr \rangle
\end{equation}
and whose trace is the {\em mean curvature} $\HHH= \mathrm{tr} \,  \BBB= g^{ij} h_{ij}$ of the hypersurface (with these choices, the standard sphere of $\R^n$ has positive mean curvature).

\begin{rem}\label{graph}
If the hypersurface $M\subseteq\T^n$ is the graph of a function $f:U\to\R$ with $U$ an open subset of $\R^{n-1}$, that is, $\varphi(x)=(x,f(x))$, then we have 
$$
g_{ij}=\delta_{ij}+\frac{\partial f}{\partial x_i}\frac{\partial f}{\partial x_j}\,,\qquad\qquad \nu=-\frac{(\nabla f,-1)}{\sqrt{1+\vert\nabla f\vert^2}}
$$
$$
h_{ij}=-\frac{{\mathrm {Hess}}_{ij}f}{\sqrt{1+\vert\nabla f\vert^2}}
$$
$$
\HHH=-\frac{\Delta f}{\sqrt{1+\vert\nabla f\vert^2}}
+\frac{{\mathrm {Hess}}f(\nabla f,\nabla f)}{\big(\sqrt{1+\vert\nabla f\vert^2}\big)^3}=-
\Div\biggl(\frac{\nabla f}{\sqrt{1+\vert\nabla f\vert^2}}\biggr)
$$
where ${\mathrm {Hess}}f$ is the Hessian of the function $f$.
\end{rem}

Then, the following {\em Gauss--Weingarten relations} hold,
\begin{equation}\label{GW}
\frac{\pa^2\varphi}{\pa x_i\pa
 x_j}=\Gamma_{ij}^k\frac{\pa\varphi}{\pa
 x_k}- h_{ij}\nu\qquad\qquad\frac{\pa\nu}{\pa x_j}=
h_{jl}g^{ls}\frac{\pa\varphi}{\pa x_s}\,,
\end{equation}
which easily imply $\vert\nabla\nu\vert=\vert\BBB\vert$ and the identity
\begin{equation}
\Delta\varphi=g^{ij}\Bigl(\frac{\partial^2\varphi}{\partial x_i\partial
 x_j}-\Gamma_{ij}^k\frac{\partial\varphi}{\partial
 x_k}\Bigr)=-g^{ij}h_{ij}\nu=-\HHH\nu\,.\label{lap}
\end{equation}

The Riemann tensor is expressed via the second fundamental form as follows ({\em Gauss equations}),
\begin{equation}\label{Gauss-eq}
\RRR_{ijkl}\,=\,h_{ik}h_{jl}-h_{il}h_{jk}
\end{equation}
hence, the formulas for the interchange of covariant derivatives,
which involve the Riemann tensor, become
\begin{align}
\nabla_i\nabla_jX^s-\nabla_j\nabla_iX^s=&\,\RRR_{ijkl}g^{ks}X^l=\RRR_{ijl}^sX^l=
\big(h_{ik}h_{jl}-h_{il}h_{jk}\big)g^{ks}X^l\\
\nabla_i\nabla_j\omega_k-\nabla_j\nabla_i\omega_k=&\,\RRR_{ijkl}g^{ls}\omega_s=\RRR_{ijk}^s\omega_s=
\big(h_{ik}h_{jl}-h_{il}h_{jk}\big)g^{ls}\omega_s\label{ichange}
\end{align}
for every vector field $X$ and $1$--form $\omega$.

\medskip

We say that a set $E\subseteq\T^n$ is a {\em smooth set} if it is the closure of an open subset of $\T^n$ and its boundary $\partial E$ is a smooth embedded hypersurface (unless otherwise stated all the sets we are going to consider will be smooth). Then, for a smooth set $E \subseteq \T^n$ and $\eps>0$ small enough, we let ($d$ is the ``Euclidean'' distance on $\T^n$) 
\begin{equation}
\label{tubdef}
N_\eps=\{x \in \T^n \, : \, d(x,\pa E)<\eps\}
\end{equation}
to be a {\em tubular neighborhood} of $\pa E$ such that the {\em orthogonal projection map} $\pi_E:N_\eps\to \pa E$ giving the (unique) closest point on $\pa E$ and the {\em signed distance function} $d_E:N_\eps\to\R$ from $\pa E$
\begin{equation}\label{sign dist}
d_E(x)=
\begin{cases}
d(x, \pa E) &\text{if $x \notin E$}\\
-d(x, \pa E) &\text{if $x \in E$}
\end{cases} 
\end{equation}
are well defined and smooth in $N_\eps$ (for a proof of the existence of such tubular neighborhood and of all the subsequent properties, see~\cite{ManMen} for instance). Moreover, for every $x\in N_\eps$, the projection map is given explicitly by 
\begin{equation}\label{eqcar2050}
\pi_E(x)=x-\nabla d^2_E(x)/2=x-d_E(x)\nabla d_E(x)
\end{equation}
and the unit vector $\nabla d_E(x)$ is orthogonal to $\pa E$ at the point $\pi_E(x)\in\partial E$, indeed actually 
\begin{equation}\label{eqcar410bis}
\nabla d_E(x)=\nabla d_E(\pi_E(x))=\nu(\pi_E(x))\,,
\end{equation}
which means that the integral curves of the vector field $\nabla d_E$ are straight segments orthogonal to $\pa E$.\\ 
This clearly implies that the map 
\begin{equation}\label{eqcar410}
\partial E\times (-\eps,\eps)\ni (y,t)\mapsto L(y,t)=y+t\nabla d_E(y)=y+t\nu(y)\in N_\eps
\end{equation}
is a smooth diffeomorphism with inverse
$$
N_\eps\ni x\mapsto L^{-1}(x)=(\pi_E(x),d_E(x))\in \partial E\times (-\eps,\eps)\,.
$$
Moreover, denoting with $JL$ its Jacobian (relative to the hypersurface $\pa E$), there holds
\begin{equation}\label{eqcar411}
0<C_1\leqslant J L(y,t)\leqslant C_2
\end{equation}
on $\pa E\times(-\eps,\eps)$, for a couple of constants $C_1,C_2$, depending on $E$ and $\eps$.

\medskip 
{\em From now on, in all the rest of the work, with $N_\eps$ we will always denote a suitable tubular neighborhood of a smooth set, with the above properties.}
\medskip

By means of such tubular neighborhoods of smooth sets $E\subseteq\T^n$, we can speak of ``$W^{k,p}$--closedness'' (or of ``$C^{k}$--closedness'') of sets. Indeed, fixed a smooth set $E$, we say that $F,F'\subseteq\T^n$ are $\delta$--close in $W^{k,p}$ (or in $C^{k}$) for some $\delta>0$ ``small enough'' if, denoted by $F\triangle F'=(F \cup F')\setminus (F \cap F')$ the symmetric difference between $F$ and $F'$, we have $\vol(F\triangle F') < \delta$ and that $\partial F,\partial F'$ are contained in a tubular neighborhood $N_\eps$ of $E$ as above, described by 
$$
\pa F=\{y+\psi(y)\nu_E(y) \, : \, y\in \pa E\}\qquad\text{ and }\qquad\pa F'=\{y+\psi'(y)\nu_E(y) \, : \, y\in \pa E\},
$$
for two functions $\psi:\partial E\to\R$ with $\norma{\psi-\psi'}_{W^{k,p}(\pa E)}< \delta$ (respectively, $\norma{\psi-\psi'}_{C^{k}(\pa E)}< \delta$). That is, we are asking that the two sets $F$ and $F'$ differ by a set of small Lebesgue measure and that their boundaries are ``close'' in $W^{k,p}$ (or $C^{k}$) as graphs on $\partial E$.

Moreover, we can define the following families of hypersurfaces.
\begin{definition}\label{psiofF}
Given a smooth set $E\subseteq\T^n$ and a tubular neighborhood $N_\eps$ of $\pa E$, as in formula~\eqref{tubdef}, for any $M<\eps$, we denote by $\mathfrak{C}^1_M(E)$, the class of all sets $F\subseteq E\cup N_\eps$ such that $\vol(F\triangle E)\leqslant M$ and 
\begin{equation}\label{front}
\pa F=\{y+\psi_F(y)\nu_E(y):\, y\in \pa E \}\,,
\end{equation}
for some function $\psi_F\in C^1(\pa E)$, with $\norma{\psi_F}_{C^1(\pa E)}\leqslant M$ (hence, $\pa F\subseteq N_\eps$).
\end{definition}

\begin{definition}\label{convF}
Given a sequence of smooth sets $F_i\in\mathfrak{C}^{1}_M(E)$, for some smooth set $E\subseteq\T^n$, we will write $F_i\to F$ in $W^{k,p}$ if there exists $F\in\mathfrak{C}^1_M(E)$ such that for every $\delta>0$, if $i\in\N$ is large enough there holds $\vol(F_i\triangle F) < \delta$ and, describing the boundaries of $F_i,F$ as
$$
\pa F_i=\{y+\psi_i(y)\nu_E(y) \, : \, y\in \pa E\}\qquad\text{ and }\qquad\pa F=\{y+\psi(y)\nu_E(y) \, : \, y\in \pa E\},
$$
for some smooth function $\psi_i,\psi:\partial E\to\R$, we have $\norma{\psi_i-\psi}_{W^{k,p}(\pa E)}< \delta$.
\end{definition}

\section{The surface diffusion flow and strictly stable critical sets}\label{sec3}

In all the following $\T^n=\R^n/\!\!\sim$ is a flat $n$--dimensional torus, quotient of $\R^n$ by a discrete group of translations generated by some $n$ linearly independent vectors. Since we want to deal with the surface diffusion flow of embedded smooth hypersurfaces which are boundaries of smooth sets (recall that any of them is the closure of an open subset of $\T^n$), we give the following definition. 

\begin{definition}\label{def:SDsol}
Let $E\subseteq \T^n$ be a smooth set. We say that the family $E_t\subseteq\T^n$, for $t \in[0,T)$ with $E_0=E$, is a {\em surface diffusion flow} starting from $E$ if the map $t\mapsto\rchi_{E_t}$ is continuous from $[0,T)$ to $L^1(\T^n)$ and the hypersurfaces $\partial E_t$ move by surface diffusion, that is, there exists a smooth family of embeddings $\varphi_t:\partial E\to \T^n$, for $t \in[0,T)$, with $\varphi_0=\mathrm{Id}$ and $\varphi_t(\partial E)=\partial E_t$, such that
\begin{equation}\label{sdf2-0}
\frac{\partial\varphi_t}{\partial t}=(\Delta\HHH)\nu\,,
\end{equation}
where, at every point and time, $\HHH$ and $\Delta$ are respectively the mean curvature and the Laplacian (with the Riemannian metric induced by $\T^n$, that is, by $\R^n$) of the moving hypersurface $\pa E_t$, while $\nu$ is the ``outer'' normal to the smooth set $E_t$.
\end{definition}

\begin{remark} An alternative way to describe the flow is to speak of the sets ``enclosed'' by the boundary hypersurfaces moving by surface diffusion. This anyway would introduce an ambiguity, since every hypersurface $\partial E_t$ clearly ``separate'' $\T^n$ in components and one should indicate which ones are actually the sets $E_t$ at every time $t$. The use of the continuity of the map $t\mapsto\rchi_{E_t}$ is a way to avoid such ambiguity. Moreover, it follows easily that being the solution of the PDE system~\eqref{sdf2-0} unique, by Theorem~\ref{th:EMS0}, the sets $E_t$ are uniquely determined (being a ``geometric flow'', actually the same ``geometric'' uniqueness also holds for the hypersurfaces $\partial E_t$, like for the mean curvature flow, see~\cite[Section~1.3]{Manlib}).
\end{remark}

Then, we have the following proposition, which is actually Theorem~\ref{th:EMS0} ``adapted'' to the above definition.

\begin{proposition}\label{th:EMS1}
Given a smooth set $E\subseteq\T^n$ and a tubular neighborhood $N_\eps$ of $\pa E$, as in formula~\eqref{tubdef} and $M<\varepsilon$, for every $E_0\subseteq\T^n$ smooth set in $\mathfrak{C}^1_{M}(E)$, whose boundary $\pa E_0$ is represented by
\begin{equation}\label{front0}
\pa E_0=\{y+\psi_0(y)\nu_E(y):\, y\in \pa E \}
\end{equation}
for a smooth function $\psi_0:\pa E\to \R$, there exists a unique surface diffusion flow $E_t$, starting from $E_0$, determined by
\begin{equation}\label{front00}
\pa E_t=\{y+\psi_t(y)\nu_E(y):\, y\in \pa E \}
\end{equation}
with smooth functions $\psi_t:\pa E\to\R$, for $t$ in some maximal positive interval of time $[0,T(E_0))$, with $T(E_0)$ depending on the $C^{2,\alpha}$--norm of $\psi_0$.
\end{proposition}

We now define the {\em strictly stable critical sets} and we state our main theorem, we refer to~\cite[Section~2]{DDMsurvey} for the following facts.

We consider the {\em Area functional} $\A(\partial E)$ on the family of smooth sets $E\subseteq\T^n$, giving the $(n-1)$--dimensional ``area'' of the boundary of $E$, with a constraint on the ($n$--dimensional) volume $\vol(E)$. It is then well known that a set $E$ is a {\em critical set} (that is, with zero {\em constrained first variation}) if and only if its boundary satisfy $\HHH=\lambda$, for some constant $\lambda\in\R$. This is obtained, by ``testing'' the first variation of the Area functional with all the {\em volume--preserving variations} of $E$, which turns out to have as infinitesimal generators, vector fields $X$ on $\partial E$ satisfying $X=\psi\nu_E$ with $\psi\in C^{\infty}(\pa E)$ such that $\int_{\pa E} \psi \, d\mu =0$.\\
Then, at a critical set $E$, the second variation of the volume--constrained Area functional along such vector fields $X=\psi\nu_E$ on $\partial E$ is given by
\begin{equation}
\frac{d^2 }{dt^2}\A(\pa E_t)\Bigl|_{t=0} =\int_{\pa E}\bigl(|\nabla \psi|^2- \psi^2|\BBB|^2\bigr)\, d\mu\,.
\end{equation}
This motivates the following definition.

\begin{definition}\label{Pi}
Given any smooth open set $E\subseteq \T^n$, we define the quadratic form 
\beq\label{quadraticform}
\Pi_E(\psi)= \int_{\pa E} \big (\abs{\grad \psi}^2 - \psi^2\abs{\BBB}^2 \big ) \, \dmu \, ,
\eeq
for all $\psi$ in the space of Sobolev functions
$$
\widetilde H^1 (\pa E) = \Bigl\{ \psi \in H^1(\pa E) \, : \, \int_{\pa E} \psi \, \dmu =0 \Bigr\}\,.
$$
\end{definition}

\smallskip

{\em From now on we will extensively use Sobolev spaces on smooth hypersurfaces. Most of their properties hold as in $\R^n$, standard references are~\cite{AdamsFournier} in the Euclidean space and~\cite{aubin0} on manifolds.}

\medskip

We notice that for every $\eta\in \R^n$, we have $\langle\eta\,|\,\nu_E\rangle\in \widetilde H^1 (\pa E)$, as in general, for every smooth set $E\subseteq\T^n$, there holds
\beq\label{divnorm}
\int_{\pa E} \langle\eta\,|\,\nu_E\rangle \, \dmu=\int_{E} \Div\eta\,dx=0\,,\qquad\text{ hence also}\qquad
\int_{\pa E} \nu_E\, \dmu=0\,.
\eeq
Then, setting $E_t=E+t\eta$, by the translation invariance of the area functional $\A$, we have $\A(\partial E_t)=\A(\partial E)$, thus
\begin{equation}
0=\frac{d^2}{dt^2} \A(\partial E_t) \Bigr \vert_{t=0}= \Pi_E (\langle\eta\,|\,\nu_E\rangle)\,,
\end{equation}
that is, the form $\Pi_E$ is zero on the vector subspace 
\begin{equation}\label{T}
T(\pa E) = \bigl\{\langle\eta\,|\,\nu_E\rangle\, : \, \eta\in \R^n \bigr\}\subseteq \Htilde^1(\pa E)
\end{equation}
of dimension clearly less than or equal to $n$ (and at least one).
Then, we split 
\begin{equation}\label{decomp}
\Htilde^1(\pa E) = T(\pa E) \oplus \Tort(\pa E)\,,
\end{equation}
where $\Tort (\pa E)\subseteq \Htilde^1(\pa E)$ is the vector subspace $L^2$--orthogonal to $T(\pa E)$ (with respect to the measure $\mu$ on $\partial E$), that is,
\begin{align}
\Tort(\pa E)
=&\,\Bigl \{\psi \in \Htilde^1(\pa E) \, : \, \int_{\pa E} \psi \nu_E
 \, \dmu = 0\Bigr \}\\
=&\,\Bigl \{\psi \in H^1(\pa E) \, : \,\int_{\pa E} \psi \, \dmu = 0 \,\,\text{ and }\, \int_{\pa E} \psi \nu_E \, \dmu = 0\Bigr \}\label{Tort} 
\end{align}
and we give the following ``stability'' conditions.

\begin{definition}\label{str stab}
We say that a critical set $E \subseteq\T^n$ for $\A$ under a volume constraint is {\em stable} if
\begin{equation}\label{stabile}
\Pi_E (\psi) \geqslant 0 \qquad\text{for all $\psi \in \Htilde^1(\pa E)$}
\end{equation}
and {\em strictly stable} if actually
\begin{equation}\label{strettstabile}
\Pi_E (\psi) > 0 \qquad\text{for all $\psi \in \Tort(\pa E) \setminus \{0 \}$.}
\end{equation}
\end{definition}

As one can easily guess, these stability notions are related to the (sufficient and necessary) local minimality properties of a set for the volume--constrained Area functional, as it is shown in~\cite{AcFuMoJu} (see also the discussion in~\cite[Section~2.2]{DDMsurvey}).

We observe that there exists an orthonormal frame $\{e_1, \dots, e_n \}$ of $\R^n$ such that
\begin{equation}\label{ort}
\int_{\pa E} \langle \nu_E | e_i \rangle \langle \nu_E | e_j \rangle \, \dmu=0,
\end{equation}
for all $i \ne j$, indeed, considering the symmetric $n\times n$--matrix $A= (a_{ij})$ with components $a_{ij}= \int_{\pa E} \nu_E^i \nu_E^j \, \dmu$, where $\nu_E^i=\langle\nu_E|\eps_i\rangle$ for some basis $\{\eps_1,\dots,\eps_n\}$ of $\R^n$, we have
\begin{equation}
\int_{\pa E} (O\nu_E)_i (O \nu_E)_j \, \dmu = (OAO^{-1})_{ij},
\end{equation}
for every $O \in SO(n)$. Choosing $O$ such that $OAO^{-1}$ is diagonal and setting $e_i=O^{-1}\eps_i$, relations~\eqref{ort} are clearly satisfied. Hence, considering such basis, the functions $\langle \nu_E | e_i \rangle$ which are not identically zero are an orthogonal basis of $T(\pa E)$ and we set 
\begin{equation}\label{IIeq}
\II_E=\bigl\{i\in\{1,\dots,n\}\,:\,\text{$\langle\nu_E|e_i\rangle$ is not identically zero}\bigr\},\quad
\OO_E=\mathrm{Span}\{e_i\,:\,i\in \II_E\}.
\end{equation}

We observe that it is easy to see (by a dilation/contraction argument) that any strictly stable smooth critical set must be connected, but actually, Theorem~\ref{existence2} can be clearly applied also to finite unions of boundaries of strictly stable critical sets. Moreover, by the very definition above, if $\pa E$ in $\T^n$ is composed by flat pieces, hence its second fundamental form $\BBB$ is identically zero, the set $E$ is critical and stable and with a little effort, actually strictly stable. It is a little more difficult to show that any ball in any dimension $n\in\N$ is strictly stable (it is obviously a critical set), which is connected to the study of the eigenvalues of the Laplacian on the sphere $\SSS^{n-1}$, see~\cite[Theorem~5.4.1]{groemer}, for instance. Then, the same holds for all the ``cylinders'' $\R^k\times\SSS^{n-k-1}\subseteq\R^n$, bounding $E\subseteq\T^n$ after taking their quotient by the same equivalence relation defining $\T^n$, determined by the standard integer lattice of $\R^n$.

If $n=2$, it follows that the only bounded strictly stable critical sets of the {\em Length} functional (that is, the Area functional in $2$--dimensional ambient space) in the plane are the disks and in $\T^2$ they are the disks and the ``strips'' with straight borders. In the three--dimensional case, a first classification of the smooth stable ``periodic'' critical sets for the volume--constrained Area functional, was given by Ros in~\cite{Ros}, where it is shown that in the flat torus $\T^3$, they are {\em balls}, {\em $2$--tori}, {\em gyroids} or {\em lamellae}, as in the following figure.

\begin{figure}[H]
\centering
\includegraphics[scale=0.6]{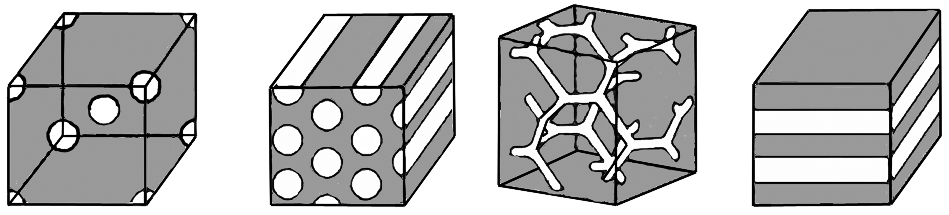}
\caption{Periodic critical set classified in~\cite{Ros}: balls, $2$--tori, gyroids and lamellae (from left to right).}
\label{figurauno}
\end{figure} Lamellae, $2$--tori and balls are actually strictly stable, while in~\cite{Gr,GrWo,Ross} the authors established the strict stability of gyroids only in some cases. 

\medskip

We are going to prove in Theorem~\ref{existence2} that if a smooth set is ``close'' enough to a strictly stable critical set for the volume constrained Area functional, then its surface diffusion flow exists for every time and smoothly converges to a ``translate'' of such set.

\section{Evolution of geometric quantities and basic estimates}

Along any surface diffusion flow $\varphi_t:M\to\T^n$ (or when the ambient is a general flat space) we have the following evolution equations (computed in detail in~\cite[Proposition~3.4]{mant5} for a general geometric flow of hypersurfaces),
\begin{equation}
\frac{\partial}{\partial t}g_{ij}=2\Delta\HHH h_{ij}\,,\quad\quad\frac{\partial}{\partial t}g^{ij}
=- 2\Delta\HHH h^{ij}\, ,\quad\quad\frac{\partial}{\partial t}\mu=\HHH \Delta\HHH\mu\label{derg}
\end{equation}
and
\begin{equation}
\frac{\partial}{\partial t}\Gamma_{jk}^i={\nabla\BBB} * \Delta\HHH+\BBB * \nabla \Delta \HHH\label{derCh}
\end{equation}
where $T * S$ (following Hamilton~\cite{hamilton1}) denotes a tensor formed by a sum of terms each one of them obtained by the product of a real constant with the contraction on some indices of the pair $T$, $S$ with the metric $g_{ij}$ and/or its inverse $g^{ij}$. Just to give an example, if $T$ is a $(0,3)$--tensor and $S$ a $(2,2)$--tensor, then one possible contribution to $T*S$ is given by 
$$
g^{pq}T_{ijp}S^{kl}_{mq}+5g^{pq}T_{pij}S^{kl}_{qm}\,.
$$
A very useful property of such $*$--product is that
$|T*S|\leqslant C|T||S|$ where the constant $C$ depends only on the ``algebraic structure'' of $T*S$, moreover, it clearly holds $\nabla (T*S)=\nabla T*S+T*\nabla S$.\\
Then, arguing as in~\cite[Proposition~2.3.1]{Manlib}, we get the following evolution equation for the mean curvature
\begin{equation}
\frac{\partial\,}{\partial t}\HHH =-\Delta\Delta\HHH -\Delta\HHH\vert\BBB\vert^2\label{derH}
\end{equation}
(notice that this equation further highlights the fourth--order nature of the flow).

We now introduce some notation which will be useful for the computations that follow (see~\cite{mant5}). If $T_1, \dots , T_l$ is a finite family of tensors (here $l$ is not an index of the tensor $T$), 
with the symbol
$$
\bbigstar_{i=1}^l T_i
$$
we will mean $T_1 * T_2 * \dots * T_l$.\\
With the symbol $\pol_{s}(\nabla^\alpha T,\nabla^\beta S,\dots,\nabla^\gamma R)$ we will denote a ``polynomial''
in the tensors $T,S,\dots,R$ and their iterated covariant derivatives with the $*$ product as 
\beq \label{pols}
\pol_{s}(\nabla^\alpha T,\nabla^\beta S,\dots,\nabla^\gamma R)\,\,=\sum_{i+j+\dots+k=s}
c_{ij\dots k}\,\nabla^{i} T * \nabla^{j}S*\dots * \nabla^{k} R
\eeq
where the $c_{ij\dots k}$ are some real constants and $i\leqslant\alpha$, $j\leqslant\beta$, ... , $k\leqslant\gamma$. Moreover, we set $\pol_0(\,\cdot\,)=0$. Notice that every tensor must be present in every additive 
term of $\pol_{s}(\nabla^\alpha T,\nabla^\beta S,\dots,\nabla^\gamma R)$ and there are no repetitions.\\
If $\alpha=2$, $\beta=1$, $\gamma=0$ and $s=2$, an example of such a polynomial is given by
\begin{equation}
\pol_2(\grad^2T,\grad S,R)=7g^{pq}(\grad_p\grad_qT_{ij})S_{kl}R_{mn}+5g^{pq} (\grad_pT_{ij})(\grad_q S_{kl})R_{mn}.
\end{equation}
We will use instead the symbol $\qol^s(\nabla^\alpha\BBB,\nabla^\beta\HHH)$ for a completely contracted ``polynomial'' (hence a function) of the iterated covariant derivatives of $\BBB$ and $\HHH$, respectively up to $\alpha$ and $\beta$ (repetitions are allowed), where in every additive term both $\BBB$ and $\HHH$ must be present and $\HHH$ without derivatives is considered as a contracted $\BBB$--factor. That is, 
\beq \label{qols}
\qol^s(\nabla^\alpha\BBB, \nabla^\beta\HHH)=
\sum\,\bbigstar_{k=1}^{p}\nabla^{i_k}\BBB\,\bbigstar_{l=1}^q\nabla^{j_l}\HHH
\eeq
with $p$, $q\geqslant1$, $i_1,\dots,i_p\leqslant\alpha$ and $1\leqslant j_1,\dots,j_q\leqslant\beta$, then the coefficient $s$ denotes the sum
\beq \label{qols2}
s=\sum_{k=1}^p (i_k+1)\,+\,\sum_{l=1}^q (j_l+1)\,.
\eeq

\medskip

{\em We advise the reader that in the following the ``polynomials'' $\pol_s$ and $\qol^s$ could
vary from a line to another in a computation, by addition of ``similar'' terms.} 

\medskip

With this notation, we have the following ``computation'' lemmas.

\begin{lem}\label{eulerkey}
For every tensor $T$ and function $f$ on $M$, we have
\begin{equation}
\frac{\partial\,}{\partial t}\nabla^sT=\nabla^s\frac{\partial T}{\partial t}+\pol_{s}(\nabla^{s-1}T,\nabla^{s}\BBB,\nabla^{s} \Delta\HHH)\quad\text{ for every $s\geqslant 1$} \label{scambioderivateT}
\end{equation}
\begin{equation}
\frac{\partial\,}{\partial t}df=d\frac{\partial f}{\partial t}\qquad\text{ and }\qquad \frac{\partial\,}{\partial t}\nabla^sf=\nabla^s\frac{\partial f}{\partial t}+\pol_{s-1}(\nabla^{s-2}(\nabla f),\nabla^{s-1}\BBB,\nabla^{s-1} \Delta\HHH)\label{scambioderivatef}
\end{equation}
for every $s\geqslant 2$.
\end{lem}
\begin{proof} We show the first equation by induction on $s\in\N$. If $s=1$, we have
\begin{align*}
\frac{\partial\,}{\partial t}\nabla T=&\,
\frac{\partial\,}{\partial t}(\partial T+T\Gamma)=\frac{\partial\,}{\partial t}\partial T+\frac{\partial\,}{\partial t}(T\Gamma)
=\partial \frac{\partial T}{\partial t}+\frac{\partial T }{\partial t}\Gamma+T\frac{\partial \Gamma}{\partial t}\\
=&\,\nabla\frac{\partial T}{\partial t}+T*{\nabla\BBB} * \Delta\HHH+T*\BBB * \nabla \Delta \HHH=\nabla\frac{\partial T}{\partial t}+\pol_1(T,{\nabla\BBB},\nabla \Delta \HHH)\,,
\end{align*}
where we computed ``schematically'', denoting with $\partial$ the standard derivative in coordinates (with commute with 
$\frac{\partial\,}{\partial t}$) and with $\Gamma$ the Christoffel symbols, moreover, we used formula~\eqref{derCh}.\\
Now, assuming that formula~\eqref{scambioderivateT} holds up to $s-1$, we apply it to the tensor $S=\nabla T$
\begin{align*}
\frac{\partial\,}{\partial t}\nabla^s T=&\,\frac{\partial\,}{\partial t}\nabla^{s-1} S=\nabla^{s-1}\frac{\partial S}{\partial t}+\pol_{s-1}(\nabla^{s-2}S,\nabla^{s-1}\BBB,\nabla^{s-1} \Delta\HHH)\\
=&\,\nabla^{s-1}\frac{\partial\,}{\partial t}\nabla T+\pol_{s}(\nabla^{s-1}T,\nabla^{s-1}\BBB,\nabla^{s-1} \Delta\HHH)\\
=&\,\nabla^{s-1}\Bigl(\nabla\frac{\partial T}{\partial t}+\pol_1(T,{\nabla\BBB},\nabla \Delta \HHH)\,\Bigr)+\pol_{s}(\nabla^{s-1}T,\nabla^{s-1}\BBB,\nabla^{s-1} \Delta\HHH)\\
=&\,\nabla^{s}\frac{\partial T}{\partial t}+\pol_s(\nabla^{s-1}T,\nabla^{s}\BBB,\nabla^{s} \Delta\HHH)
\end{align*}
by the properties of the $*$--product. Hence, formula~\eqref{scambioderivateT} is proved.\\
To get equation~\eqref{scambioderivatef}, we apply the previous formula to $T=\nabla f$ as follows
\begin{align*}
\frac{\partial\,}{\partial t}\nabla^sf=&\,\frac{\partial\,}{\partial t}\nabla^{s-1}\nabla f=\nabla^{s-1}\frac{\partial}{\partial t}\nabla f+\pol_{s-1}(\nabla^{s-2}(\nabla f),\nabla^{s-1}\BBB,\nabla^{s-1} \Delta\HHH)\\
=&\,\nabla^{s}\frac{\partial f}{\partial t}+\pol_{s-1}(\nabla^{s-2}(\nabla f),\nabla^{s-1}\BBB,\nabla^{s-1} \Delta\HHH)
\end{align*}
and we are done.
\end{proof}
\begin{remark}
In the following, we will denote by $\BBB(\grad f, \grad f)$ the second fundamental form applied to the gradient of a function $f$. In local coordinates this reads as
$$
\BBB(\grad f, \grad f)= g^{ik} g^{jl} h_{kl} \grad_i f \grad_j f\, .
$$
\end{remark}
We are now ready to compute the evolution equations and then estimate the key quantities for our flow in dimension $n=4$.

\begin{proposition}\label{computeenvar}
Let $E_t\subseteq \T^4$ be a surface diffusion flow. Then, the following equations hold
\begin{align}
\frac{d}{dt} \int_{\pa E_t} \abs{\grad \HHH}^2 \, \dmu_t =&\,- 2 \Pi_{E_t} ( \Delta \HHH) + \int_{\pa E_t} \HHH \Delta \HHH \abs{\grad \HHH}^2 \, \dmu_t\label{energyfinale1}\\
&\,-\int_{\pa E_t} 2 \BBB(\grad \HHH, \grad \HHH) \Delta \HHH \, \dmu_t\\
\frac{d}{dt} \int_{\pa E_t} \abs{\grad^2\HHH}^2 \, \dmu_t =&\,- 2 \int_{\pa E_t} |\nabla^4\HHH|^2 \, \dmu_t+\int_{\pa E_t} \qol^{10}(\BBB,\nabla^3\HHH)\,\dmu_t\label{energyfinale}\\
&\,+\int_{\pa E_t} \qol^{10}(\nabla \BBB,\nabla^4\HHH)\,\dmu_t
\end{align}
where $\Pi_{E_t}$ is the quadratic form defined in formula~\eqref{quadraticform} and
\begin{itemize}
\item every ``monomial'' of $\qol^{10}(\BBB,\nabla^3\HHH)$ has $4$ factors in $\BBB$, $\nabla\HHH$ and their covariant derivatives, the factor $\BBB$ (or $\HHH$, without derivatives) is present exactly one time and the other three factors are derivatives of $\nabla\HHH$ up to $\nabla^3\HHH$, with $\grad^{3} \HHH$ or $\grad^{2} \HHH$ present at least one time;
\item every ``monomial'' of $\qol^{10}(\nabla \BBB,\nabla^4\HHH)$ has $4$ factors in $\BBB$, $\nabla\BBB$ and covariant derivatives of $\HHH$, the factor $\BBB * \BBB$ or $\BBB * \nabla \BBB$ is present exactly one time, the other two factors are derivatives of $\nabla\HHH$ up to $\nabla^4\HHH$. The factor $\grad ^4 \HHH$ is present exactly one time, with the exception of ``monomials'' of kind $\grad^3 \HHH * \BBB^2 * \grad^3 \HHH$.
\end{itemize}
Finally, the coefficients of these ``polynomials'' are {\em algebraic}, that is, they are the result of formal manipulations, in particular, they are independent of $E_t$.
\end{proposition}
\begin{proof}
Taking into account the evolution equations~\eqref{derg} and~\eqref{derH}, we compute
\begin{align}
\frac{d}{dt} \int_{\pa E_t} \abs{\nabla\HHH}^2 \, \dmu_t
=&\, \int_{\pa E_t}\HHH\abs{\nabla\HHH}^2 \,\Delta\HHH\,\dmu_t
-\int_{\pa E_t} 2h^{ij}\nabla_i\HHH\nabla_j\HHH\,\Delta\HHH\, d\mu_t\\
&\,-\int_{\pa E_t} 2g^{ij}\nabla_i\HHH\nabla_j\bigl(\abs{\BBB}^2\Delta\HHH+\Delta \Delta\HHH\bigr)\, d\mu_t\\
=&\, \int_{\pa E_t}\HHH\abs{\nabla\HHH}^2 \,\Delta\HHH\,\dmu_t-\int_{\pa E_t} 2\BBB(\nabla\HHH,\nabla\HHH)\,\Delta\HHH\, d\mu_t\\
&\,+\int_{\pa E_t} 2\abs{\BBB}^2(\Delta\HHH)^2\, d\mu_t
+\int_{\pa E_t} 2\Delta\HHH\,\Delta \Delta\HHH\, d\mu_t\\
=&\, \int_{\pa E_t}\HHH\abs{\nabla\HHH}^2 \,\Delta\HHH\,\dmu_t-\int_{\pa E_t} 2\BBB(\nabla\HHH,\nabla\HHH)\,\Delta\HHH\, d\mu_t\\
&\,+\int_{\pa E_t} 2\abs{\BBB}^2(\Delta\HHH)^2\, d\mu_t
-\int_{\pa E_t} 2\vert\nabla\Delta\HHH\vert^2\, d\mu_t\,,
\end{align}
where the first term on the right comes from the area variation and the second from the evolution equation of the inverse of the metric (see formulas~\eqref{derg}). Then, we have formula~\eqref{energyfinale1}, recalling the definition of $\Pi_{E_t}$ in~\eqref{quadraticform}.

To get equation~\eqref{energyfinale} we compute analogously
\begin{align}
\frac{d}{dt} \int_{\pa E_t} \abs{\nabla^2\HHH }^2 \, \dmu_t =&\, \int_{\pa E_t} \abs{\nabla^2\HHH}^2\HHH \Delta \HHH \, \dmu_t + 2\int_{\pa E_t} g\Bigl(\nabla^2\HHH,\frac{\pa}{\pa t}\nabla^2\HHH\Bigr)\, \dmu_t\\
&\,- \int_{\pa E_t} 2h^{ij}g^{kl} \Delta \HHH \grad^2_{ik}\HHH \grad^2_{jl}\HHH \, \dmu_t\\
&\,- \int_{\pa E_t} 2h^{kl}g^{ij}\Delta \HHH \grad^2_{ik}\HHH \grad^2_{jl}\HHH \, \dmu_t\label{conto1} \, .
\end{align}
We focus on the second integral, noticing that we can collect all the other terms inside the integrals in a ``polynomial'' $\qol^{10}(\BBB,\nabla^1(\nabla\HHH))$ such that every ``monomial'' has $4$ factors in $\BBB$, $\nabla\HHH$ and their covariant derivatives (remember that we consider $\HHH$ as a contracted $\BBB$--factor, in the first term -- we will always do the same also in the following) and at least three of them are derivatives of $\nabla\HHH$.\\
By formula~\eqref{scambioderivatef} in Lemma~\ref{eulerkey} with $f=\HHH$ and $s=2$, we have
\begin{align*}
\frac{\partial\,}{\partial t}\nabla^2\HHH=&\,\nabla^2\frac{\partial}{\partial t}\HHH+\pol_{1}(\nabla^0(\nabla\HHH),\nabla^1\BBB,\nabla^1 \Delta\HHH)\\
=&\,\nabla^2\big(-\Delta\Delta\HHH -\Delta\HHH\vert\BBB\vert^2\big)+\pol_{1}(\nabla^0(\nabla\HHH),\nabla^1\BBB,\nabla^1 \Delta\HHH)
\end{align*}
hence, the second integral in formula~\eqref{conto1} is equal to
\begin{align}
\int_{\pa E_t} g\Bigl(\nabla^2\HHH\,\frac{\pa}{\pa t}\nabla^2\HHH\Bigr)\, \dmu_t
=&\,\int_{\pa E_t} g\big(\nabla^2\HHH,\nabla^2\big(-\Delta\Delta\HHH -\Delta\HHH\vert\BBB\vert^2\big) \big)\,\dmu_t\\
&\,+\int_{\pa E_t} g\big(\nabla^2\HHH,\pol_{1}(\nabla^0(\nabla\HHH),\nabla^1\BBB,\nabla^1 \Delta\HHH)\big)\,\dmu_t\\
=&\,-\int_{\pa E_t} g^{ij}g^{kl}\nabla^4_{ikjl}\HHH\,\Delta\Delta\HHH\,\dmu_t\\
&\,-\int_{\pa E_t} g^{ij}g^{kl}\nabla^4_{ikjl}\HHH\Delta\HHH\vert\BBB\vert^2\,\dmu_t\\
&\,+\int_{\pa E_t} g\big(\nabla^2\HHH,\pol_{1}(\nabla^0(\nabla\HHH),\nabla^1\BBB,\nabla^1 \Delta\HHH)\big)\,\dmu_t\,,\label{conto2}
\end{align}
where we integrated by parts twice.
\\Then, recalling the properties of $\pol_{1}(\nabla^0(\nabla\HHH),\nabla^1\BBB,\nabla^1 \Delta\HHH)$, integrating by parts in the last integral, we can ``take away'' the derivative from $\BBB$ (in the ``monomials'' containing it) and ``move'' it on the other three factors, which are derivatives of $\HHH$. Hence, such integral becomes a term of kind $\int_{\pa E_t} \qol^{10}(\BBB,\nabla^3\HHH)\,\dmu_t$, noticing that $\nabla^4\HHH$ cannot appear, as by the properties of $\pol_1(\nabla^0(\nabla\HHH),\nabla^1\BBB,\nabla^1 \Delta\HHH)\big)$ either it contains $\nabla\BBB$ or $\nabla\Delta\HHH$, but not both together in any of its ``monomials''. 
Then, recalling equations~\eqref{conto1},~\eqref{conto2} and noticing that the first, third and fourth integrals in the right hand side of equation~\eqref{conto1} is of the form 
\beq
\int_{\pa E_t} \qol^{10}(\BBB,\nabla^2\HHH)\,\dmu_t \, ,
\eeq
 we can write
\begin{align}
\frac{d}{dt} \int_{\pa E_t} \abs{\nabla^2\HHH }^2 \, \dmu_t =&-2\int_{\pa E_t} g^{ij}g^{kl}\nabla^4_{ikjl}\HHH\,\Delta\Delta\HHH\,\dmu_t
-2\int_{\pa E_t} g^{ij}g^{kl}\nabla^4_{ikjl}\HHH\Delta\HHH\vert\BBB\vert^2\,\dmu_t
\\&+\int_{\pa E_t} \qol^{10}(\BBB,\nabla^3\HHH)\,\dmu_t\label{conto3}
\end{align}
where every ``monomial'' of $\qol^{10}(\BBB,\nabla^3\HHH)$ has $4$ factors in $\BBB$, $\nabla\HHH$ and their covariant derivatives, moreover
\begin{itemize}
\item the factor $\BBB$ (or $\HHH$, without derivatives) is present exactly one time,
\item the other three factors are derivatives of $\nabla\HHH$ up to $\nabla^3\HHH$, with $\grad^{3} \HHH$ or $\grad^{2} \HHH$ present at least one time.
\end{itemize}
Now we deal with the first  term in the right hand side of equation~\eqref{conto3}. We recall that
\beq
\int_{\pa E_t} g^{ij}g^{kl}\nabla^4_{ikjl}\HHH\,\Delta\Delta\HHH\,\dmu_t=\int_{\pa E_t} g^{ij}g^{kl}g^{ms}g^{pq}\nabla_i\nabla_k\nabla_j\nabla_l\HHH\, \nabla_m\nabla_s\nabla_p\nabla_q\HHH\,\dmu_t\, .
\eeq
Then, since every interchange in a pair of subsequent covariant derivatives produces an extra ``error term'' of the form 
 $\grad^l(\Riem*\nabla^{4-l}\HHH)=\nabla^l (\BBB^2*\nabla^{4-l}\HHH)$, for $l=0,\dots, 3$, by the Gauss equations~\eqref{Gauss-eq} (while instead we can switch freely two derivatives on $\HHH$, being the Hessian symmetric), 
 we obtain 
\begin{align}
\int_{\pa E_t} g^{ij}g^{kl}\nabla^4_{ikjl}\HHH\,\Delta\Delta\HHH\,\dmu_t=&\,\int_{\pa E_t} g^{ij}g^{kl}g^{ms}g^{pq}\nabla^4_{qsjl} \HHH\, \nabla^4_{ikmp}\HHH\,\dmu_t\\
&\,+ \sum_{l=0}^{3}\int_{\pa E_t}\nabla^{2}\HHH*\nabla^l (\BBB^2*\nabla^{4-l}\HHH)\,\dmu_t\,. \\
\label{intmetrica}
\end{align}

Then, we notice that, integrating twice by parts in every integral in the sum above with $l=2,3$ and only one time when $l=1$, the right hand term in~\eqref{intmetrica} is equal to
\begin{align*}
-2\int_{\pa E_t} |\nabla^4\HHH|^2\,\dmu_t&\,+\int_{\pa E_t}\nabla^{2}\HHH*\BBB^2*\nabla^{4}\HHH\,\dmu_t+\sum_{l=2}^{3}\int_{\pa E_t}\nabla^{4}\HHH*\nabla^{l-2} (\BBB^2*\nabla^{4-l}\HHH)\,\dmu_t\\
&\,+\int_{\pa E_t}\nabla^{3}\HHH*\BBB^2*\nabla^{3}\HHH\,\dmu_t\,,
\end{align*}
hence, the last two integrals on the first line contain the factor $\nabla^{4} \HHH$ exactly one time and we can finally write 
\beq\label{conto4}
-2\int_{\pa E_t} |\nabla^4\HHH|^2\,\dmu_t+\int_{\pa E_t} \qol^{10}(\nabla \BBB,\nabla^4\HHH)\,\dmu_t +\int_{\pa E_t}\nabla^{3}\HHH*\BBB^2*\nabla^{3}\HHH\,\dmu_t\,,
\eeq
where every ``monomial'' of $\qol^{10}(\nabla\BBB,\nabla^4\HHH)$ (where we collect all the ``error terms'' in such manipulation) has $4$ factors in $\BBB$, $\nabla\BBB$ and covariant derivatives of $\HHH$, moreover
\begin{itemize}
\item the factor $\BBB^2$ or $\BBB * \nabla\BBB$ is present exactly one time,
\item the other two factors are derivatives of $\nabla\HHH$ up to $\nabla^4\HHH$,
\item the factor $\grad ^4 \HHH$ is present exactly one time, with the exception of ``monomials'' of kind $\grad^3 \HHH * \BBB^2 * \grad^3 \HHH$.
\end{itemize}
Finally, noticing that the remaining term in formula~\eqref{conto3} is of the form 
$$
\int_{\pa E_t} \qol^{10}(\nabla \BBB,\nabla^4\HHH) \, \dmu_t\, ,
$$
putting together equations~\eqref{conto3} and~\eqref{conto4}, we get the second formula of the proposition.
\end{proof}

\smallskip

{\em In all the following, we will be interested in having uniform estimates for the families of sets in $\mathfrak{C}^{1}_{M_E}(E)$, given a smooth set $E\subseteq\T^n$ and a tubular neighborhood $N_\eps$ of $\pa E$ as in formula~\eqref{tubdef}, for $M_E<\eps$. To this aim, we need that the constants in the Sobolev, Poincar\'e, Gagliardo--Nirenberg interpolation and Calder\'on--Zygmund inequalities relative to all the hypersurfaces $\pa F$ boundaries of the sets $F\in\mathfrak{C}^{1}_{M_E}(E)$, are uniform. This is proved in detail in~\cite{DDM} (for the Calder\'on--Zygmund inequalities, we actually need that $F\in\mathfrak{C}^{1}_{M_E}(E)$, with $M_E>0$ small enough), hence, from now on we will use the adjective ``uniform'' in order to underline such fact. We also highlight that in all the following we will denote with $C$ a constant which may vary from a line to another and depends only on $E$ and $M_E$.}

\smallskip

\begin{proposition}[Gagliardo--Nirenberg interpolation inequalities]\label{interptensor}
Let $E \subseteq \T^n$ be a smooth set, $j$, $m$ be integer such that $0 \leqslant j <m $ and $0<r,q \leqslant +\infty$. Then, for every $F\in\mathfrak{C}^{1}_{M_E}(E)$ and every covariant tensor $T=T_{i_1 \dots i_l}$ the following ``uniform'' interpolation inequalities hold
\begin{equation}\label{gn1}
 \Vert\nabla^j T\Vert_{L^{p}{(\pa F)}}\leqslant\,C\,\big(\Vert 
\nabla^m T\Vert_{L^{r}{(\pa F)}}+\Vert T\Vert_{L^{r}{(\pa F)}}\big)^{\theta}\Vert
 T\Vert_{L^q{(\pa F)}}^{1-\theta}\,,
\end{equation}
with the compatibility condition
\beq\label{comp1}
\frac{1}{p}=\frac{j}{n-1}+\theta \Big (\frac{1}{r}-\frac{m}{n-1}\Big )+\frac{1-\theta}{q}\,,
\eeq
for all $\theta \in[j/m,1)$ for which $p\in[1,+\infty)$ is nonnegative and where $C$ is a constant depending only on $n$, $j$, $m$, $p$, $q$, $r$ and $E$, $M_E$. Moreover, if $f:\partial F\to\R$ is a smooth function, inequality~\eqref{gn1} becomes
\beq\label{gn2}
\norma{\grad^j f}_{L^p(\partial F)} \leqslant C\,\norma{\grad^m f}_{L^r(\pa F)}^{\theta} \norma{f}_{L^q(\pa F)}^{1-\theta}
\eeq
if $j \geqslant 1$ or $j=0$ and $\fint_{\partial F}f\,d\mu=0$.
By density, these inequalities clearly extend to functions and tensors in the appropriate Sobolev spaces.
\end{proposition}
\begin{proof}[Proof -- Sketch]
For a single fixed regular hypersurface $\partial F$, inequality~\eqref{gn2} is given by Theorem~3.70 in~\cite{aubin0}, while inequality~\eqref{gn1} for $T$ equal to a function $f:\partial F\to\R$ can be obtained by repeating step by step the proof of such theorem, once established the standard Sobolev inequality for hypersurfaces without boundary,
\beq\label{sob}
\Vert f\Vert_{L^{p^*}{(\pa F)}}\leqslant\,C\,\big(\Vert\nabla f\Vert_{L^{p}{(\pa F)}}+\Vert f\Vert_{L^{p}{(\pa F)}}\big)\,,
\eeq
for every $p\in[1,n-1)$, where $p^*=np/(n-p)$ (an example of such argument can be found in~\cite[Section~6]{mant5}).\\
The extension of inequality~\eqref{gn1} to tensors can be obtained as in~\cite[Sections~5 and~6]{mant5}, by means of the estimate
(see~\cite{aubin0}, Proposition~2.11 and also~\cite{cantor1,cantor2}),
$$
\left\vert \,\nabla \sqrt{\vert T\vert^2 + \varepsilon^2}\,\right\vert\,=\,
\left\vert \,\frac{\langle \nabla T , T\rangle}{\sqrt{\vert T\vert^2 +
 \varepsilon^2}}\, \right\vert
\,\leqslant\,\frac{\vert T \vert}{\sqrt{\vert T\vert^2 +
 \varepsilon^2}}\,\vert\nabla T\vert\,\leqslant \,\vert\nabla T\vert
$$
clearly leading to the previous Sobolev inequality for tensors, as $\sqrt{\vert T\vert^2 + \varepsilon^2}$ converges to $\vert T \vert$ when $\varepsilon\to0$ (this argument is necessary as $\vert T \vert$ is not necessarily smooth).\\
Finally, the ``uniformity'' in the constants of the inequalities, independently of $F\in\mathfrak{C}^{1}_{M_E}(E)$, follows by the same independence in the Sobolev inequalities (by checking the proof of Theorem~3.70 in~\cite{aubin0}). This is shown and discussed in detail in~\cite{DDM}.
\end{proof}

\begin{remark}\label{poincare}
Notice that in the same hypotheses of this proposition, by inequality
\beq\label{poinc0}
\Vert f-\overline{f\,}\Vert_{L^{q^*}{(\pa F)}}\leqslant\,C\,\Vert\nabla f\Vert_{L^{q}{(\pa F)}}\,,
\eeq
we also have the following ``uniform'' Poincar\'e inequalities
\beq\label{poinc1}
\Vert f-\overline{f\,}\Vert_{L^{p}{(\pa F)}}\leqslant\,C\,\Vert\nabla f\Vert_{L^{p}{(\pa F)}}\,,
\eeq
for every $p\in[1,+\infty]$.
\end{remark}

\begin{remark}\label{rem1}
Very similar ``uniform'' interpolation inequalities are worked out in~\cite{mant5}, for any family of smooth, $n$--dimensional, regular hypersurfaces $N\subseteq\R^{n+1}$ satisfying $\vol(N) +\Vert\HHH\Vert_{L^{n+\delta}(N)}\leqslant C$, for some $\delta>0$, instead of being boundaries of sets belonging to $\mathfrak{C}^{1}_{M_E}(E)$.
\end{remark}

As a direct consequence of Proposition~\ref{interptensor}, we have the following lemma.

\begin{lem}\label{interptensor2}
Let $E \subseteq \T^n$ be a smooth set and $j,m\in\N$ with $1\leqslant j<m$. Then, for every $F\in\mathfrak{C}^{1}_{M_E}(E)$ and every covariant tensor $T$, the following uniform inequalities hold, for every $\varepsilon>0$,
\begin{equation}
\norma{\grad^j T}_{L^p(\partial F)}^2 \!\leqslant\!C \norma{\grad^m T}_{L^2(\pa F)}^{2\theta} \norma{\nabla T}_{L^2(\pa F)}^{2(1-\theta)}+ C\norma{\nabla T}_{L^2(\pa F)}^2\! \leqslant \varepsilon\norma{\grad^m T}_{L^2(\pa F)}^2\!+ C\norma{\nabla T}_{L^2(\pa F)}^2\label{gn3bis}
\end{equation}
with the compatibility condition
\beq\label{comp2}
\frac{1}{p}=\frac{j-1}{n-1}-\theta\Big(\frac{m-1}{n-1}\Big)+\frac{1}{2}\,,
\eeq
for all $\theta \in\big[\frac{j-1}{m-1},1\big)$ for which $p\in[1,+\infty)$ is nonnegative.\\
The constants $C$ depends only on $n$, $j$, $m$, $p$, $E$, $M_E$ and $\varepsilon$.
\end{lem}
\begin{proof}
The first inequality in formula~\eqref{gn3bis} comes from inequality~\eqref{gn1}, by substituting $\nabla T$ in place of $T$, while the second one follows by Young inequality.
\end{proof}

{\em From now on, we consider as ambient space the four--dimensional flat torus $\T^4$.}

\begin{proposition}\label{estimateenvar}
Let $E_t\subseteq \T^4$ be a surface diffusion flow such that $E_t\in\mathfrak{C}^{1}_{M_E}(E)$, for some smooth set $E$. Then,
\begin{align}
\frac{d}{dt} \int_{\pa E_t} \abs{\grad \HHH}^2 \, \dmu_t \leqslant & \,- 2 \Pi_{E_t} ( \Delta \HHH) + \eps \norma{\nabla^{4}\HHH}^2_{L^2(\pa E_t)}+C_1\big(1+\norma{\grad \HHH}_{L^2(\pa E_t)}^{\tau}\big)\norma{\grad \HHH}_{L^2(\pa E_t)}^2\label{energy0}\\
\frac{d}{dt} \int_{\pa E_t} \abs{\grad^2\HHH}^2 \, \dmu_t \leqslant &\,- \norma{\nabla^4\HHH}^2_{L^2(\pa E_t)}+C_2\big(1+\norma{\grad \HHH}_{L^2(\pa E_t)}^{\tau'}\big)\norma{\grad \HHH}_{L^2(\pa E_t)}^2\label{energy1}
\end{align}
for any $\eps>0$, some exponents $\tau, \tau'>0$ and constants $C_1,C_2$ depending on $E$, $M_E$, $\eps$, $\|\BBB\|_{L^\infty(\pa E_t)}$ and $\|\nabla\BBB\|_{L^6(\pa E_t)}$.
\end{proposition}
\begin{proof}
To get the first inequality, we start estimating the second and third terms in formula~\eqref{energyfinale1} as follows,
\beq
C\int_{\pa E_t}|\BBB| \abs{\grad \HHH}^2 \abs{\grad^2 \HHH}\, \dmu_t\leqslant
C\int_{\pa E_t}|\BBB|\,\prod_{l=1}^3|\nabla^{j_l}\HHH|\, \dmu_t
\leqslant C\norma{\BBB}_{L^{\infty}(\pa E_t)} \prod_{l=1}^3 \norma{\grad^{j_l}\HHH}_{L^{\beta_l}(\pa E_t)}\,,
\eeq
where we used H\"older inequality, with exponents $\beta_l= \frac{7}{j_l+1}>2$, noticing that since $\sum _{l=1}^3 j_l=4$, we have
\beq
\sum_{l=1}^3 \frac{1}{\beta_l} =\sum_{l=1}^3 \frac{j_l+1}{7} =1\,.
\eeq
Then, by the uniform interpolation inequalities~\eqref{gn3bis}, we get
\begin{equation}
\norma{\grad^{j_l}\HHH}_{L^{\beta_l}(\pa E_t)} \leqslant C\norma{\grad^4 \HHH}^{\theta_l}_{L^2(\pa E_t)} \norma{\grad \HHH}^{1-\theta_l}_{L^{2}(\pa E_t)} +C\norma{\grad \HHH}_{L^{2}(\pa E_t)} \, , \label{interpgrad2bis}
\end{equation}
with
\beq
\theta_l=\frac{j_l-1}{3}+\frac{1}{2}-\frac{1}{\beta_l}\in \Bigl(\frac{j_l-1}{3},1 \Bigr)\,,
\eeq
for some uniform constants $C$.
Hence,
\begin{align}
C\int_{\pa E_t}|\BBB| \abs{\grad \HHH}^2 \abs{\grad^2 \HHH}\, \dmu_t
\leqslant &\,C(\norma{\BBB}_{L^{\infty}(\pa E_t)}) \Big[ \norma{\grad^4 \HHH}^{\Theta}_{L^2(\pa E_t)} \norma{\grad \HHH}^{3-\Theta}_{L^{2}(\pa E_t)} \\
&\,+\sum_{l=1}^3\norma{\grad^4 \HHH}^{\Theta-\theta_l}_{L^{2}(\pa E_t)}\norma{\grad \HHH}^{3-\Theta+\theta_l}_{L^{2}(\pa E_t)}\\
&\,+ \sum_{l=1}^3\norma{\grad^4 \HHH}^{\theta_l}_{L^{2}(\pa E_t)}\norma{\grad \HHH}^{3-\theta_l}_{L^{2}(\pa E_t)}+\norma{\grad \HHH}^{3}_{L^{2}(\pa E_t)} \Big]
\end{align}
where
\begin{align}
\Theta=\sum_{l=1}^3 \theta_l= &\sum_{l=1}^3 \frac{j_l-1}{3} +\frac{3}{2}-1=\frac{5}{6} < 2\,,
\end{align}
as $\sum _{l=1}^3 j_l=4$.\\
Finally, being $\Theta$, $\Theta-\theta_l$ and $\theta_l<2$ for every $l=1,2,3$, by the Young inequality, we conclude
\beq
C\int_{\pa E_t}|\BBB| \abs{\grad \HHH}^2 \abs{\grad^2 \HHH}\, \dmu_t\leqslant\varepsilon\norma{\grad^4 \HHH}^2_{L^2(\pa E_t)} + C \big(1+\norma{\grad \HHH}^{\tau}_{L^{2}(\pa E_t)}\big) \norma{\grad \HHH}^{2}_{L^{2}(\pa E_t)}\, \label{interpgrad3-bis}
\eeq
for any $\varepsilon>0$, with $C=C(\varepsilon, \norma{\BBB}_{L^\infty(\pa E_t)})$ and $\tau>0$.\\
About the second one, recalling formula~\eqref{energyfinale}, we start dealing with the integral 
$$
\int_{\pa E_t} \qol^{10}(\BBB,\nabla^3\HHH)\,\dmu_t\,,
$$
which is a sum of integrals each one like
\beq
\int_{\pa E_t} \BBB\,\bbigstar_{l=1}^3\nabla^{j_l}\HHH\, \dmu_t\,,
\eeq
with $0 < j_l\leqslant3$, moreover, it must hold
$$
10=1+\sum_{l=1}^3(j_l+1)\qquad\text{ that is, }\qquad \sum_{l=1}^3 j_l=6\,,
$$
by formula~\eqref{qols2}. 
Hence, 
\beq
\int_{\pa E_t}\!\!|\qol^{10}(\BBB, \nabla^3 \HHH)| \, \dmu_t \leqslant C\sum \int_{\pa E_t}\Bigl|\BBB\,\bbigstar_{l=1}^3\nabla^{j_l}\HHH\,\Bigr|\, \dmu_t\leqslant C\sum \int_{\pa E_t}\!\!|\BBB|\,\prod_{l=1}^3|\nabla^{j_l}\HHH|\, \dmu_t.
\eeq
We can estimate each integral of the last sum by means of H\"older inequality, as follows
\beq\label{stimaqol1}
\int_{\pa E_t}|\BBB|\,\prod_{l=1}^3|\nabla^{j_l}\HHH|\, \dmu_t
\leqslant \norma{\BBB}_{L^{\infty}(\pa E_t)} \prod_{l=1}^3 \norma{\grad^{j_l}\HHH}_{L^{\beta_l}(\pa E_t)}
\eeq
where $\beta_l= \frac{9}{j_l+1}>2$, which clearly satisfy 
\beq
\sum_{l=1}^3 \frac{1}{\beta_l} =\sum_{l=1}^3 \frac{j_l+1}{9} =1\,.
\eeq
Then, we now estimate any of such products as follows: applying the uniform interpolation inequalities~\eqref{gn3bis} to $\HHH$, we get
\begin{equation}
\norma{\grad^{j_l}\HHH}_{L^{\beta_l}(\pa E_t)} \leqslant C\norma{\grad^4 \HHH}^{\theta_l}_{L^2(\pa E_t)} \norma{\grad \HHH}^{1-\theta_l}_{L^{2}(\pa E_t)}+ C\norma{\grad \HHH}_{L^{2}(\pa E_t)} \, , \label{interpgrad2}
\end{equation}
for some constants $C$ depending on $\beta_l$ and coefficients $\theta_l$ which are given by
\beq\label{thetal}
\theta_l=\frac{1}{3} \Bigl ( j_l - \frac{3}{\beta_l} + \frac{1}{2}\Bigr)\in \biggl(\frac{j_l-1}{3},1 \biggr)\,.
\eeq
Then, noticing that 
\begin{align}
\Theta=\sum_{l=1}^3 \theta_l= &\frac{1}{3} \sum_{l=1}^3 j_l-\sum_{l=1}^3 \frac{1}{\beta_l}
+\frac{1}{2}=1+\frac{1}{2}<2\, ,
\end{align}
multiplying together inequalities~\eqref{interpgrad2} and applying the Young inequality as above, we have 
\begin{align}
\int_{\pa E_t}|\BBB|\,\prod_{l=1}^3|\nabla^{j_l}\HHH|\, \dmu_t
\leqslant&\, C\norma{\BBB}_{L^{\infty}(\pa E_t)}\Big[\norma{\grad^4 \HHH}^{\Theta}_{L^{2}(\pa E_t)}\norma{\grad \HHH}^{3-\Theta}_{L^{2}(\pa E_t)}\\
&\,+\sum_{l=1}^3\norma{\grad^4 \HHH}^{\Theta-\theta_l}_{L^{2}(\pa E_t)}\norma{\grad \HHH}^{3-\Theta+\theta_l}_{L^{2}(\pa E_t)}\\
&\,+ \sum_{l=1}^3\norma{\grad^4 \HHH}^{\theta_l}_{L^{2}(\pa E_t)}\norma{\grad \HHH}^{3-\theta_l}_{L^{2}(\pa E_t)}+\norma{\grad \HHH}^{3}_{L^{2}(\pa E_t)} \Big]\\
\leqslant&\, \varepsilon\norma{\grad^4 \HHH}^2_{L^2(\pa E_t)} +C \big( 1+\norma{\grad \HHH}^{\tau'}_{L^{2}(\pa E_t)}\big)\norma{\grad \HHH}^{2}_{L^{2}(\pa E_t)}\,,  \label{interpgrad3}
\end{align}
for any $\varepsilon>0$, with $C=C(\varepsilon, \norma{\BBB}_{L^\infty(\pa E_t)})$ and $\tau'>0$.\\
Hence, we conclude (by choosing appropriately $\varepsilon>0$ for each summand in $\qol^{10}(\BBB,\nabla^3\HHH)$) that
\begin{equation}
\int_{\pa E_t} \qol^{10}(\BBB,\nabla^3\HHH)\,\dmu_t\leqslant\frac{1}{2}{\norma{\grad^4 \HHH}^2_{L^2(\pa E_t)}} + C \big( 1+\norma{\grad \HHH}^{\tau'}_{L^{2}(\pa E_t)}\big)\norma{\grad \HHH}^{2}_{L^{2}(\pa E_t)}\label{interpgrad10}
\end{equation}
with $C=C(\norma{\BBB}_{L^\infty(\pa E_t)})$ and $\tau'>0$.

\smallskip

Now we deal with the last integral in formula~\eqref{energyfinale}, that is,
\beq\label{excepmon}
\int_{\pa E_t} \qol^{10}(\nabla \BBB, \nabla^4 \HHH) \, \dmu_t=  \int_{\pa E_t}\nabla^{3}\HHH*\BBB^2*\nabla^{3}\HHH\,\dmu_t+ \sum
\int_{\pa E_t} \nabla^i(\BBB^2) * \nabla^j\HHH*\nabla^4\HHH\, \dmu_t \,,
\eeq
where $0\leqslant i\leqslant1$ and $0 < j\leqslant 3$ such that (by formula~\eqref{qols2})
$$
10=i+2+j+1+5\qquad\text{ that is, }\qquad i+j=2\,.
$$
This actually implies that in the sum we have only two types of terms,
\beq
\int_{\pa E_t} \BBB^2 * \nabla^2\HHH*\nabla^4\HHH\, \dmu_t\qquad\text{ and }\qquad \int_{\pa E_t} \BBB*\nabla\BBB * \nabla\HHH*\nabla^4\HHH\, \dmu_t\,.\label{intest}
\eeq
After ``carrying'' the modulus inside the integrals and using the properties of the $*$--product, we estimate the first integral in~\eqref{intest} by means of Young inequality and inequality~\eqref{gn3bis} (with $(j,m,p)=(2,4,2)$)  
, that is
\begin{align}
\int_{\pa E_t} |\BBB^2| |\nabla^2\HHH| |\nabla^4\HHH|\, \dmu_t & \leqslant \, \varepsilon\norma{\grad^4 \HHH}^2_{L^2(\pa E_t)} + C\norma{\grad^{2}\HHH}_{L^{2}(\pa E_t)}^2 
\\& \, \leqslant 2\varepsilon\norma{\grad^4 \HHH}^2_{L^2(\pa E_t)} + C\norma{\grad \HHH}_{L^{2}(\pa E_t)}^2 \label{interpgrad4}
\end{align}
for any $\varepsilon>0$, with $C=C(\varepsilon, \norma{\BBB}_{L^\infty(\pa E_t)})$.\\
\\Analogously, applying the H\"older inequality in the second integral in~\eqref{intest}, we have
\begin{align}
\int_{\pa E_t} |\BBB| |\nabla\BBB| |\nabla\HHH| |\nabla^4\HHH|\, \dmu_t & \leqslant\varepsilon\norma{\grad^4 \HHH}^2_{L^2(\pa E_t)} + C\int_{\pa E_t}|\nabla\BBB|^2| \nabla\HHH|^2\, \dmu_t 
\\& \leqslant \,\varepsilon\norma{\grad^4 \HHH}^2_{L^2(\pa E_t)} + C\norma{\grad\BBB}^{2}_{L^6(\pa E_t)} \norma{\grad \HHH}^{2}_{L^{3}(\pa E_t)}
\\& \leqslant \,\varepsilon\norma{\grad^4 \HHH}^2_{L^2(\pa E_t)} + C \norma{\grad \HHH}^{2}_{L^{3}(\pa E_t)}
\end{align}
for some constant $C=C(\varepsilon, \norma{\BBB}_{L^\infty(\pa E_t)},\norma{\nabla\BBB}_{L^6(\pa E_t)})$.
Using again inequality~\eqref{gn3bis} (with $(j,m,p)=(1,4,3)$) 
we obtain 
\beq
\norma{\grad \HHH}^{2}_{L^{3}(\pa E_t)} \leqslant  \varepsilon\norma{\grad^4 \HHH}^2_{L^2(\pa E_t)} + C \norma{\grad \HHH}^2_{L^{2}(\pa E_t)}
\eeq
that is, 
\beq \label{interpgrad7}
\int_{\pa E_t} |\BBB| |\nabla\BBB| |\nabla\HHH| |\nabla^4\HHH|\, \dmu_t \leqslant 2 \varepsilon\norma{\grad^4 \HHH}^2_{L^2(\pa E_t)} + C \norma{\grad \HHH}^2_{L^{2}(\pa E_t)}
\eeq
where $C=C(\varepsilon, \norma{\BBB}_{L^\infty(\pa E_t)},\norma{\nabla\BBB}_{L^6(\pa E_t)})$.
Finally, by means of inequality~\eqref{gn3bis} (with $(j,m,p)=(3,4,2)$), 
we estimate the integrals of the {\em exceptional} ``monomials'' (that is, the first term in the right hand side of~\eqref{excepmon})
\beq \label{interpgrad8}
\int_{\pa E_t}|\nabla^{3}\HHH| |\BBB|^2|\nabla^{3}\HHH|\,\dmu_t\leqslant C \norma{\grad^3 \HHH}^2_{L^2(\pa E_t)} \leqslant  \,\varepsilon\norma{\grad^4 \HHH}^2_{L^2(\pa E_t)} + C \norma{\grad \HHH}^{2}_{L^{2}(\pa E_t)} \, ,
  \eeq
$C=C(\varepsilon, \norma{\BBB}_{L^\infty(\pa E_t)})$.

By means of estimates~\eqref{interpgrad4},~\eqref{interpgrad7} and~\eqref{interpgrad8}, we conclude (by choosing appropriately $\varepsilon>0$ for each summand in $\qol^{10}(\nabla \BBB,\nabla^4\HHH$) as before), that
\begin{equation}
\int_{\pa E_t} \qol^{10}(\nabla\BBB,\nabla^4\HHH)\,\dmu_t\leqslant\frac{1}{2}\norma{\grad^4 \HHH}^2_{L^2(\pa E_t)} + C\norma{\grad \HHH}^2_{L^{2}(\pa E_t)} \label{interpgrad11}
\end{equation}
with $C=C(\norma{\BBB}_{L^\infty(\pa E_t)},\norma{\nabla\BBB}_{L^6(\pa E_t)})$.

Hence, we obtain inequality~\eqref{energy1} from inequalities~\eqref{interpgrad10} and~\eqref{interpgrad11}.
\end{proof}

\begin{remark}\label{rem2}
Recalling Remark~\ref{rem1}, in the proof of this proposition we could alternatively uniformly control the constants in the interpolation inequalities by a function of $\vol(\pa E_t) +\Vert\HHH\Vert_{L^4(\pa E_t)}$ (of $\vol(\pa E_t) +\Vert\HHH\Vert_{L^n(\pa E_t)}$ in the $n$--dimensional case), instead of using Proposition~\ref{interptensor}, as it is done in~\cite{mant5}. It follows that this proposition holds also for only immersed (not boundaries of sets) smooth hypersurfaces moving by the surface diffusion flow with equibounded volumes.
\end{remark}

\begin{lem}\label{w42conv}
Let $E\subseteq\T^4$ be a smooth set and $N_\eps$ be a tubular neighborhood of $\pa E$, as in formula~\eqref{tubdef}. For $M_E$ small enough and $\delta>0$, there exists a constant $C=C(E,M_E,\delta)$ such that if $F\in\mathfrak{C}^{1}_{M_E}(E)$ with
\begin{equation}
\pa F= \{y + \psi_F (y) \nu_E(y) \, : \, y \in \pa E\}
\end{equation}
for a smooth function $\psi_F:\partial E\to\R$ and
\beq\label{ex-de02}
\int_{\pa F} |\nabla^2\HHH|^2\,\dmu + \int_{\pa F} |\nabla \HHH|^2\,\dmu \leqslant \delta\,,
\eeq
there hold
\begin{equation}
\|\BBB\|_{L^\infty(\pa F)}+\|\nabla\BBB\|_{L^{6}(\pa F)}\leqslant C\qquad\text{ and }\qquad
\|\psi_F\|_{W^{4,2}(\pa E)}\leqslant C\,.\label{eqcar50001}
\end{equation}
Moreover, for every $1\leqslant p<6$, there exists a monotone nondecreasing function $\omega:\R^+\to\R^+$, depending only on $E$ and $M_E$, with $\lim_{\delta\to0^+}\omega(\delta)=0$ and such that if $F$ satisfies the further condition
$$
\vol(F\triangle E)\leqslant \delta\,,
$$
then $\Vert\psi_F\Vert_{W^{3,p}(\pa E)}\leqslant\omega(\delta)$.\\
As a consequence, if $E_i\subseteq \mathfrak{C}^{1}_{M_E}(E) $ is a sequence of smooth sets such that 
$$
\sup_{i\in\N}\,\int_{\pa E_i}\abs{\nabla^2\HHH}^2\, d\mu_i+\int_{\pa E_i}\abs{\nabla\HHH}^2\, \dmu_i<+\infty\,,
$$
then there exists a (non necessarily smooth) set $E'\in \mathfrak{C}^{1}_{M_E}(E)$ such that, up to a (non relabeled) subsequence, $E_i\to E'$ in $W^{3,p}$ as $i\to\infty$, for all $1\leqslant p<6$. Moreover, if
$$
\int_{\pa E_i}|\nabla^2\HHH|^2\, d\mu_i+\int_{\pa E_i}|\nabla\HHH|^2\, d\mu_i \to 0\,,
$$
as $i\to\infty$, the set $E'$ is critical for the volume--constrained Area functional, that is, its mean curvature is constant.
\end{lem}
\begin{proof}
Let $F\in\mathfrak{C}^{1}_{M_E}(E)$ with an associate function $\psi_F:\pa E\to\R$ as in the statement. We start by observing that, by the uniform Sobolev inequality~\eqref{sob} (extended to tensors, as for the other inequalities in Proposition~\ref{interptensor}) applied to $\nabla\HHH$, we have
\beq\label{sob2}
\Vert \nabla\HHH\Vert_{L^{6}{(\pa F)}}\leqslant\,C\,\big(\Vert\nabla^2\HHH\Vert_{L^{2}{(\pa F)}}+\Vert\nabla\HHH\Vert_{L^{2}{(\pa F)}}\big)\leqslant C\sqrt{\delta}
\eeq 
then, by the Sobolev embedding (when $p$ is larger than the dimension $n-1=3$), we get
\beq\label{sob3}
\Vert \HHH\Vert_{L^{\infty}{(\pa F)}}\leqslant\,C\,\big(\Vert\nabla\HHH\Vert_{L^{6}{(\pa F)}}+\Vert\HHH\Vert_{L^{2}{(\pa F)}}\big)\leqslant C\sqrt{\delta}+C\Vert\HHH\Vert_{L^{2}{(\pa F)}}
\eeq
\beq\label{sob4}
\Vert \HHH-\overline{\HHH}\Vert_{L^{\infty}{(\pa F)}}\leqslant\,C\,\Vert\nabla\HHH\Vert_{L^{6}{(\pa F)}}\leqslant C\sqrt{\delta}
\eeq 
where $\overline{\HHH}=\fint_{\partial F}\HHH\,d\mu$ and all the constants depends only on $E$ and $M_E$.\\
By the uniform $C^{1}$--bounds on $\pa F$, we may find a finite family (only depending on $E$ and $M_E$) of ``solid'' cylinders of the form $\mathcal C_k = D_k+\nu_E(x_k)\R$, with $D_k\subseteq T_{x_k}E$ a closed disk of fixed radius $R>0$ centered at the origin, for a finite family of points $x_k\in E$, such that $\partial F\cap \mathcal C_k$ is the graph on $D_k$ of a smooth function $f_k:D_k\to\R$, with
\begin{equation}\label{normaC1alpha}
\|f_k\|_{C^{1}(D_k)}\leqslant M_E
\end{equation}
for every $k$ and $\pa F=\bigcup \partial F\cap \mathcal C_k$.\\
Since we want to estimate $\int_{\pa F\cap\mathcal C_k}\HHH\,d\mu$, which is a ``geometric'' quantity, we can assume (by means of an isometry) that $T_{x_k}E=\langle e_1,e_2,e_3\rangle$, hence $\nu_E(x_k)=e_4$, in the canonical orthonormal basis of $\R^4$ and 
$$
\pa F\cap\mathcal C_k=\{(x,f_k(x))\, : \, x\in D_k\}\,.
$$
Then, by formulas in Remark~\ref{graph} we have 
$$
\HHH=-\Div\biggl(\frac{\nabla f_k}{\sqrt{1+\vert\nabla f_k\vert^2}}\biggr)\, ,
$$
hence, 
\begin{align}
\int_{D_k}\HHH\, dx=&\,-\int_{D_k}\Div\bigg (\frac{\nabla f_k}{\sqrt{1+|\nabla f_k|^2}}\bigg )\, dx=-\int_{\pa D_k}\biggl\langle\frac{\nabla f_k}{\sqrt{1+|\nabla f_k|^2}}\,\bigg\vert\,\frac{x}{|x|}\bigg\rangle\, d\sigma\\
=&\,\int_{\pa D_k}\bigg\langle\nu_F\,\bigg\vert\,\frac{x}{|x|}\bigg\rangle\, d\sigma\label{intHn}
\end{align}
where $\sigma$ is the canonical (standard) $2$--dimensional measure on the sphere $\pa D_k$. Thus, being the last term at most equal to the area of the sphere $\pa D_k$, we get
\begin{equation}
\overline{\HHH}\,\vol(D_k)=\int_{D_k}\!\!\!(\overline{\HHH}-\HHH)\, dx+\int_{D_k}\HHH\, dx\leqslant\int_{D_k}\!\!\!|\HHH-\overline{\HHH}|\, dx+ C\leqslant C\int_{\partial F\cap\mathcal C_k}\!\!\!\!|\HHH-\overline{\HHH}|\, dx+ C
\end{equation}
where in the last inequality we kept into account estimate~\eqref{normaC1alpha} in changing the domain (and variables) of integration. Hence, controlling the last term of this inequality by estimate~\eqref{sob4}, it follows that $\overline{\HHH}$ is bounded by a constant depending on $E$, $M_E$, $\delta$ and the same then holds also for $\HHH$. In particular, recalling that the volume of $\partial F$ is uniformly bounded (as $F\in\mathfrak{C}^{1}_{M_E}(E)$), we have that $\HHH\in L^q(\pa F)$ for every $q\in[1,+\infty)$. Then, choosing $M_E$ small enough, Theorem~3.1 in~\cite{DDM} holds, saying that we have an analogous uniform estimate on $\BBB$ in $L^q(\pa F)$, for every $q\in[1,+\infty)$.\\
Once we have a control on $\|\BBB\|_{L^q(\pa F)}$, for some exponent $q$ larger than the dimension of the hypersurfaces, again if $M_E$ is small enough, we have the following uniform higher order Calder\'on--Zygmund--type inequalities discussed in~\cite[Section~3.1]{DDM}, holding in any dimension,
\beq\label{HkCZ}
\|\nabla^k\BBB\|_{L^2(\pa F)}\leqslant C_k\big(1+\|\nabla^k\HHH\|_{L^{2}(\pa F)}\big)
\eeq
for every $k\in\N$, where the constants $C_k$ depend on $E$, $M_E$ and $\|\BBB\|_{L^q(\pa F)}$ (and the dimension), hence in our situation they depend on $E$ and $M_E$.\\
It then follows
\beq\label{eqcar43}
\|\nabla^2\BBB\|_{L^{2}(\pa F)}\leqslant C(E,M_E, \delta)
\eeq
and, by the uniform Sobolev embeddings in dimension $3$, we conclude
\beq\label{eqcar44}
\|\BBB\|_{L^q(\pa F)}+\|\nabla\BBB\|_{L^{6}(\pa F)}\leqslant C(E,M_E,\delta)
\eeq
for every $q\in[1,+\infty)$.

These geometric estimates on $\BBB$ and their derivatives, can be ``transferred'' to estimates on the function $\psi_F:\partial E\to\R$, by means of the technique of localization--representation for any ``graphical'' hypersurface on $\pa E$ introduced by Langer in~\cite{langer2} for surfaces, generalized to any dimension by Delladio~\cite{della1} and fully developed in details by Breuning in the papers~\cite{breun1,breun2,breun3} (such technique is similar to the one we used to estimate $\overline\HHH$ above). In particular, by the results in~\cite{breun3}, under a uniform control on $\|\BBB\|_{L^q(\pa F)}$ with $q$ larger than the dimension of the hypersurface, we have that an estimate on $\|\BBB\|_{W^{k,p}(\pa F)}$ implies a uniform estimate on $\|\psi_F\|_{W^{k+2,p}(\pa F)}$ and viceversa, for every set $F\in\mathfrak{C}^{1}_{M_E}(E)$. Hence, by the previous estimates~\eqref{eqcar43} and~\eqref{eqcar44} on $\BBB$ and its derivatives, we conclude
$$
\|\psi_F\|_{W^{4,2}(\partial E)}\leqslant C(E,M_E,\delta)\,.
$$
Then, we notice that, by uniform Sobolev embeddings, we have 
$$
\|\nabla^2\psi_F\|_{L^\infty(\partial E)}\leqslant C(E,M_E,\delta)
$$
which in turn implies $\|\BBB\|_{L^\infty(\partial F)}\leqslant C(E,M_E,\delta)$, by what we said above.\\
Now, in the hypotheses of the lemma on a sequence of sets $E_i$, writing
$$
\pa E_i =\{y+\psi_i(y)\nu_E(y):\, y\in \pa E\}\,,
$$
by the previous estimates and the uniform Sobolev compact embeddings 
$$
W^{4,2}(\pa E)\hookrightarrow W^{3,p}(\pa E)\hookrightarrow C^{1}(\pa E)
$$
for all $1\leqslant p<6$, up to a (not relabeled) subsequence there exists a set $E'\in \mathfrak{C}^{1}_{M_E}(E)$ such that $\psi_i\to \psi_{E'}$ in $W^{3,p}(\pa E)$ (and in $C^{1}(\pa E)$) where
$$
\pa E' =\{y+\psi_{E'}(y)\nu_E(y):\, y\in \pa E\}\,,
$$
for all $1\leqslant p<6$.\\
If actually
$$
\int_{\pa E_i}|\nabla^2\HHH|^2\, d\mu_i+\int_{\pa E_i}|\nabla\HHH|^2\, d\mu_i \to 0\,,
$$
clearly for the limit set $E'$ the mean curvature must be constant.

The fact that $\|\psi_F\|_{W^{3,p}(\partial E)}$ goes uniformly to zero as $\delta\to 0$, hence we have a function $\omega$ as in the statement, follows by the fact that, assuming $F_i\in\mathfrak{C}^{1}_{M_E}(E)$ and
\beq
\vol(F_i\triangle E)\leqslant \delta_i\,,\qquad\qquad\int_{\pa F_i} |\nabla^2\HHH|^2\,\dmu_i + \int_{\pa F_i} |\nabla \HHH|^2\,\dmu_i \leqslant \delta_i
\eeq
with $\delta_i\to0$, as $i\to\infty$, by the previous argument we have that $\psi_{F_i}:\partial E\to\R$ converges to some $\psi:\pa E\to\R$ in $W^{3,p}(\partial E)$, hence in $L^1(\pa E)$, while the limit $\vol(F_i\triangle E)\to0$ implies that $\|\psi_{F_i}\|_{L^1(\pa E)}\to 0$, then we conclude that $\psi$ must be zero and we have the thesis.
\end{proof}

\section{Global existence and stability}

\begin{thm}\label{existence2}
Let $E\subseteq\T^4$ be a strictly stable critical set for the Area functional under a volume constraint and let $N_\eps$ be a tubular neighborhood of $\pa E$, as in formula~\eqref{tubdef}. For $M_E<\eps/2$ small enough, there exists $\delta>0$ such that, if $E_0$ is a smooth set in $\mathfrak C^{1}_{M_E}(E)$ satisfying $\vol(E_0)= \vol(E)$ and
\beq\label{eqcar777}
\vol(E_0\triangle E)\leqslant\delta\qquad\text{ and }\qquad\int_{\pa E_0} \vert \nabla^2\HHH_0\vert^2\, \dmu_0 + \int_{\pa E_0} \vert \nabla \HHH_0\vert^2\, \dmu_0 \leqslant \delta\, ,
\eeq
then, the unique smooth surface diffusion flow $E_t$ starting from $E_0$, given by Proposition~\ref{th:EMS1}, is defined for all $t\geqslant0$. Moreover, $E_t$ converges smoothly to $E'=E+\eta$ exponentially fast as $t\to +\infty$, for some $\eta\in \R^4$, with the meaning that the sequence of smooth functions $\psi_t : \pa E \to \R$ representing $\pa E_t$ as ``normal graphs'' on $\pa E$, that is,
$$
\pa E_t= \{ y+ \psi_{t} (y) \nu_{E}(y) \, : \, y \in \pa E\},
$$
satisfy, for every $k\in\N$,
$$
\Vert \psi_{t}-\psi\Vert_{C^k(\pa E)}\leqslant C_ke^{-\beta_k t}\, ,
$$
for every $t \in [0, +\infty)$, for some positive constants $C_k$ and $\beta_k$, where $\psi: \pa E\to \R$ represents $\pa E'=\pa E+\eta$ as a ``normal graph'' on $\pa E$.
\end{thm}

\begin{remark}\label{existence2++}
The request that $E_0$ belongs to $\mathfrak{C}^{1}_{M_E}(E)$ with $M_E$ small enough, is necessary only in order to be able to represent its boundary as a graph of a function with bounded gradient on $\pa E$ and to have uniform Sobolev, interpolation and Calder\'on--Zygmund inequalities, as proved in~\cite[Section~3]{DDM}, while the first condition~\eqref{eqcar777} is a ``closedness'' assumption in $L^1$ for $E_0$ and $E$ (that is, on $\psi_0$). The second ``small energy'' condition~\eqref{eqcar777} in the theorem implies (see the last part of Lemma~\ref{w42conv} and its proof) that the mean curvature of $\pa E_0$ is ``close'' to be constant, as it is for the strictly stable set $E$ (actually for any critical set). Notice that this latter is a fourth order condition for the boundary of $E_0$ and that all these assumptions are clearly implied by an appropriate $W^{4,2}$--closedness of $E_0$ to $E$, arguing as in Lemma~\ref{w42conv}.
\end{remark}

\begin{remark} In the whole paper, with a little abuse of notation, the ``translations'' in $\mathbb T^4$ (respectively, $\mathbb T^n$) are identified with vectors of $\mathbb R^4$ (respectively, $\mathbb R^n$).
\end{remark}

\begin{proof}[Proof of Theorem~\ref{existence2}]
By choosing $M_E$ small enough, we assume that for every set $F\in\mathfrak{C}^{1}_{2M_E}(E)$, all the constants in the inequalities we are going to consider for functions on $\partial F$ are uniform, depending on $E$ and $M_E$, by~\cite{DDM}.\\
After choosing some small $\delta_0>0$, we consider the surface diffusion flow $E_t$ starting from $E_0\in\mathfrak{C}^{1}_{M_E}(E)$ satisfying
\begin{equation}\label{initial}
\vol(E_0\triangle E)\leqslant\delta\qquad\text{ and }\qquad\int_{\pa E_0} |\grad^2\HHH_0|^2\, d\mu_0+ \int_{\pa E_0} |\nabla \HHH_0|^2\, d\mu_0\leqslant \delta\,,
\end{equation}
for $\delta<\delta_0/2$ and we let $T(E_0)\in (0,+\infty]$ be the maximal time such that the flow is defined for $t$ in the interval $[0,T(E_0))$, $E_t\in \mathfrak{C}^{1}_{2M_E}(E)$,
\begin{equation}\label{cond3}
\vol(E_t\triangle E)\leqslant\delta_0\qquad\text{ and }\qquad\mathcal F(t) =\int_{\pa E_t} |\grad^2\HHH|^2\, d\mu_t+ \int_{\pa E_t} |\nabla \HHH|^2\, d\mu_t\leqslant\delta_0\,.
\end{equation}
All the moving boundaries $\pa E_t$ can be represented as normal graphs on $\pa E$ as 
$$
\pa E_t= \big\{ y+ \psi_t(y) \nu_{E}(y) \, : \, y \in \pa E\big\}
$$
for some smooth functions $\psi_t:\pa E\to \R$. Moreover, if $T(E_0)<+\infty$, then at least one of the three following conditions must hold:
\begin{itemize}
\item $\limsup_{t\to T(E_0)}\|\psi_t\|_{C^{1}(\pa E)}=2M_E$
\item $\limsup_{t\to T(E_0)}{\mathcal F}(t)=\delta_0$
\item $\limsup_{t\to T(E_0)}\vol(E_t\triangle E)=\delta_0$
\end{itemize}
otherwise, restarting the flow from a time $\overline{t}$ close enough to $T(E_0)$ by means of Proposition~\ref{th:EMS1}, we have the contradiction that $T(E_0)$ cannot be the maximal time defined above. Indeed, the time interval of smooth existence of the flow given by such proposition is bounded below by a constant depending on the $C^{2,\alpha}$--norm of $\psi_{\overline{t}}$ and this latter by a constant depending on $\delta_0$, by the first point of Lemma~\ref{w42conv} and Sobolev (uniform) embeddings.\\
We are going to show that if $\delta_0$ was chosen small enough, there exists $\delta>0$ such that none of these conditions can occur, hence $T(E_0)=+\infty$, that is, the surface diffusion flow of $E_0$ exists for all time.

Let us define, for $K>2$, the following ``energy'' function
\beq\label{E(t)}
\E(t) = \int_{\pa E_t} |\nabla^2\HHH|^2\, \dmu_t+ K \int_{\pa E_t} |\nabla \HHH|^2\, \dmu_t\geqslant\mathcal F(t)
\eeq
(notice that also holds $\E(t)\leqslant K\mathcal F(t)$). From Lemma~\ref{w42conv} we easily have
\begin{equation}
\|\BBB\|_{L^\infty(\pa E_t)}+\|\grad \BBB\|_{L^6(\pa E_t)}\leqslant S(\mathcal{F}(t))\leqslant S(\mathcal{E}(t)) \,,\label{eqcar50005}
\end{equation}
for $t\in [0, T(E_0))$, where the function $S:[0,+\infty)\to\R^+$ is continuous and monotone nondecreasing and it is determined by $E$ and $M_E$.

\smallskip

We now split the rest of the proof into steps. Our first goal will be showing that the function ${\mathcal{E}}$ decreases in time if $\delta$ is small enough, for an appropriate constant $K$.

\medskip

\noindent \textbf{Step ${\mathbf 1}$} (\textit{Monotonicity of ${\mathcal E}$})\textbf{.}\\
By Proposition~\ref{estimateenvar}, for any $t\in[0,T(E_0))$, we have
\begin{align}
\frac{d}{dt} \mathcal{E}(t)\leqslant & \,- 2K \Pi_{E_t} ( \Delta \HHH) + \eps K \norma{\nabla^4\HHH}^2_{L^2(\pa E_t)}+S_1(\mathcal{E}(t))K \big(1+\norma{\grad \HHH}^{\tau}_{L^{2}(\pa E_t)}\big) \norma{\grad \HHH}^{2}_{L^{2}(\pa E_t)} \\
&- \norma{\nabla^4\HHH}^2_{L^2(\pa E_t)}+S_2(\mathcal{E}(t))\big(1+\norma{\grad \HHH}_{L^2(\pa E_t)}^{\tau'}\big)\norma{\grad \HHH}_{L^2(\pa E_t)}^2\, , \label{monotonicity1}
\end{align}
for some exponents $\tau,\tau'>0$ and continuous, monotone nondecreasing functions $S_1, S_2:[0,+\infty)\to\R^+$ (as the function $S$ above), depending on $E$, $M_E$, by inequality~\eqref{eqcar50005}.\\
Choosing $\eps= 1/2K$ in inequality~\eqref{monotonicity1}, we obtain
\begin{align}
\frac{d}{dt} \mathcal{E}(t)\leqslant & \,- 2K \Pi_{E_t} ( \Delta \HHH) + \frac{1}{2} \norma{\grad^4 \HHH}^2_{L^2(\pa E_t)} +  S_1(\mathcal{E}(t)) K \big(1+\norma{\grad \HHH}^{\tau}_{L^{2}(\pa E_t)}\big) \norma{\grad \HHH}^{2}_{L^{2}(\pa E_t)} \\
&\,- \norma{\nabla^4\HHH}^2_{L^2(\pa E_t)}+S_2(\mathcal{E}(t))\big(1+\norma{\grad \HHH}_{L^2(\pa E_t)}^{\tau'}\big)\norma{\grad \HHH}_{L^2(\pa E_t)}^2\\
\leqslant & \,- 2K \Pi_{E_t} ( \Delta \HHH) - \frac{1}{2}\norma{\nabla^4\HHH}^2_{L^2(\pa E_t)}+\big(S_1(\mathcal{E}(t))K+S_2(\mathcal{E}(t))\big)\norma{\grad \HHH}_{L^2(\pa E_t)}^2\\
& \,+S_1(\mathcal{E}(t))K\norma{\grad \HHH}_{L^2(\pa E_t)}^{2+\tau}+S_2(\mathcal{E}(t))\norma{\grad \HHH}^{2+\tau'}_{L^{2}(\pa E_t)}\label{conto1000}
\end{align}
with $\tau,\tau'>0$.
Then, by the ``uniform'' Poincar\'e inequality~\eqref{poinc1} with $f= \HHH$, that is
\beq
\Vert\HHH-\overline\HHH\Vert_{L^2(\partial E_t)}\leqslant C\Vert\nabla\HHH\Vert_{L^2(\partial E_t)} \label{Poincareunif}
\eeq
and by interpolation, we have 
$$
\Vert\nabla^2\HHH\Vert_{L^2(\partial E_t)}^2\leqslant C\Vert\nabla^4\HHH\Vert_{L^2(\partial E_t)}\Vert\HHH-\overline\HHH\Vert_{L^2(\partial E_t)}\leqslant C\Vert\nabla^4\HHH\Vert_{L^2(\partial E_t)}\Vert\nabla\HHH\Vert_{L^2(\partial E_t)}\,.
$$
Hence, using Young inequality again, we get
$$
\Vert\nabla^2\HHH\Vert_{L^2(\partial E_t)}^2\leqslant\frac{1}{2}\Vert\nabla^4\HHH\Vert_{L^2(\partial E_t)}^2+C\Vert\nabla\HHH\Vert_{L^2(\partial E_t)}^2\,,
$$
that is,
$$
-\frac{1}{2}\Vert\nabla^4\HHH\Vert_{L^2(\partial E_t)}^2\leqslant-\Vert\nabla^2\HHH\Vert_{L^2(\partial E_t)}^2+C\Vert\nabla\HHH\Vert_{L^2(\partial E_t)}^2\,.
$$
Substituting into inequality~\eqref{conto1000}, we conclude
\begin{align}
\frac{d}{dt} \mathcal{E}(t)
\leqslant & \,- 2K \Pi_{E_t} ( \Delta \HHH) - \norma{\nabla^2\HHH}^2_{L^2(\pa E_t)}+\big(S_1(\mathcal{E}(t))K+S_2(\mathcal{E}(t))+C\big)\norma{\grad \HHH}_{L^2(\pa E_t)}^2\\
& \,+S_1(\mathcal{E}(t))K\norma{\grad \HHH}_{L^2(\pa E_t)}^{2+\tau}+S_2(\mathcal{E}(t))\norma{\grad \HHH}^{2+\tau'}_{L^{2}(\pa E_t)}\\
\leqslant&\, - 2K \Pi_{E_t} ( \Delta \HHH) - \norma{\nabla^2\HHH}^2_{L^2(\pa E_t)}+ S(\mathcal{E}(t))(K+1)\norma{\grad \HHH}^{2}_{L^{2}(\pa E_t)}\\
&\,+S(\mathcal{E}(t))K\norma{\grad \HHH}_{L^2(\pa E_t)}^{2+\tau}+ S(\mathcal{E}(t))\norma{\grad \HHH}^{2+\tau'}_{L^{2}(\pa E_t)}\,,\label{conto1001}
\end{align}
with $S=\max\{S_1,S_2+C\}:[0,+\infty)\to\R^+$ continuous, monotone nondecreasing and depending on $E$ and $M_E$.\\
If we now assume that, for every $t\in[0,T(E_0))$, 
\beq
\Pi_{E_t} ( \Delta \HHH)\geqslant \sigma\norma{\grad \HHH}_{L^2(\pa E_t)}^2 \label{claim}
\eeq
for some constant $\sigma>0$, then there holds (recalling that $K>2$)
\begin{align}
\frac{d}{dt} \mathcal{E}(t)\leqslant&\, - [2K\sigma - S(\mathcal E(t)) (K+1)-2]\norma{\grad \HHH}^2_{L^{2}(\pa E_t)} - 2\norma{\nabla\HHH}^2_{L^2(\pa E_t)}\\
&\,- \norma{\nabla^2\HHH}^2_{L^2(\pa E_t)}+ S(\mathcal E(t))\norma{\grad \HHH}^{2+\tau'}_{L^{2}(\pa E_t)}+S(\mathcal E(t))K\norma{\grad \HHH}_{L^2(\pa E_t)}^{2+\tau}\\
\leqslant&\, - [2K\sigma - S(\mathcal E(t))(K+1)-2]\norma{\grad \HHH}^2_{L^{2}(\pa E_t)} 
\\&\,- 2\big(\norma{\nabla\HHH}^2_{L^2(\pa E_t)}\!+\! \norma{\nabla^2\HHH}^2_{L^2(\pa E_t)}/K\big)\\
&\,+ S(\mathcal E(t))\norma{\grad \HHH}^{2+\tau'}_{L^{2}(\pa E_t)}+S(\mathcal E(t))K\norma{\grad \HHH}_{L^2(\pa E_t)}^{2+\tau}\\
\leq &\, - P(\mathcal E(t))\norma{\grad \HHH}^2_{L^{2}(\pa E_t)}-2 \mathcal E(t)/K\\
&\,+ S(\mathcal E(t))\mathcal E(t)^{1+{\tau'}/2}+S(\mathcal E(t))K\mathcal E(t)^{1+\tau/2}
\end{align}
where $P=[2K\sigma- S(K+1)-2]:[0,+\infty)\to\R$ is continuous and monotone nonincreasing, determined by $E$ and $M_E$ and $\tau,\tau'>0$.\\
It is then an exercise of qualitative analysis of ordinary differential inequalities, to conclude that if $P(0)$ is positive and $\mathcal E(0)$ is small enough such that 
$$
-\mathcal E(0)/K+ S(\mathcal E(0))\mathcal E(0)^{1+\tau'/2}+S(\mathcal E(0))K\mathcal E(0)^{1+\tau/2}<0\,,
$$
which can be always achieved, once $K$ is fixed, the first term starts and stays negative and the ``energy'' $\mathcal E$ satisfies
\begin{equation}\label{stimadec}
\frac{d}{dt} \mathcal{E}(t)\leqslant-\mathcal{E}(t)/K
\end{equation}
for every $t\in[0,T(E_0))$, that is, the function $\mathcal{E}$ is never increasing, so it remains bounded by $\mathcal E(0)$ (moreover, it decreases exponentially and converges to zero, as $t\to+\infty$, if the flow is ``eternal''). Thus, after choosing a suitably large constant $K$, by the definition of the function $S$, it is easy to see that $P(0)>0$, if $\mathcal E(0)$ is small enough. Hence, if $\delta>0$ is small enough, since $\E(0)\leqslant K\mathcal F(0)\leqslant\delta K$, we have the monotonicity of $\mathcal E$.

\medskip

\noindent \textbf{Step ${\mathbf 2}$} (\textit{Proof of estimate~\eqref{claim}})\textbf{.}\\
We need the following key lemma which is implied by Proposition~2.35 in~\cite{DDMsurvey}, that actually simply generalizes to any dimension Lemma~2.6 in~\cite{AcFuMoJu}.

\begin{lem}\label{stimaPi}
Let $E \subseteq \T^4$ be a strictly stable critical set. For every $\theta\in (0,1]$ there exist a constant $\sigma_\theta>0$ such that, if $F \in \mathfrak{C}^{1}_{2M_E}(E)$ satisfies
\beq\label{stima000}
\vol(F\triangle E)\leqslant\delta_0\qquad\text{ and }\qquad\int_{\pa F} |\nabla \HHH|^2\,\dmu\leqslant\delta_0\,,
\eeq
for $\delta_0>0$ small enough, there holds
\begin{equation}\label{2.13}
\Pi_F(\varphi) \geqslant \sigma_\theta\norma{\varphi}^2_{L^2(\pa F)},
\end{equation}
for all $\varphi \in \Htilde^1(\pa F)$ satisfying
\begin{equation}
\min_{\eta\in\OO_E} \norma{\varphi - \langle\eta\,|\,\nu_F \rangle }_{L^2(\pa F)} \geqslant \theta \norma{\varphi}_{L^2(\pa F)}
\end{equation}
where the vector subspace $\OO_E\subseteq\R^4$ was defined in formula~\eqref{IIeq}.
\end{lem}
\begin{proof}
Representing the boundary of $F\in\mathfrak{C}^{1}_{2M_E}(E)$ as $\pa F= \{y + \psi_F (y) \nu_E(y) \, : \, y \in \pa E\}$ for a smooth function $\psi_F:\partial E\to\R$, according to Proposition~2.35 in~\cite{DDMsurvey}, fixed some $p>3$, there exists a positive constant $C=C(\theta,p)$ such that the conclusion follows if $\Vert\psi_F\Vert_{W^{2,p}(\pa E)}\leqslant C$.
This inequality follows if conditions~\eqref{stima000} hold with $\delta_0$ small enough, by the properties of the function $\omega$ stated in Lemma~\ref{w42conv} (and Sobolev embeddings).
\end{proof}

We now want to apply this lemma with $F=E_t$ and $\varphi=\Delta\HHH$, for all $t\in [0, T(E_0))$, hence, we need to show that there exists a small constant $\theta > 0$ such that
\begin{equation} \label{not a translation}
\min_{\eta\in\OO_E} \norma{ \Delta \HHH- \langle\eta\,|\,\nu_t\rangle }_{L^2(\pa E_t)}\geqslant \theta\norma{\Delta \HHH}_{L^2(\pa E_t)}\qquad\text{for all }t\in [0, T(E_0))\,.
\end{equation}
Considering the special basis $\{e_i\}$ of $\R^4$ and the associated set $i\in\II_E$ in the discussion just after Definition~\ref{str stab}, by the properties of the function $\omega$ stated in Lemma~\ref{w42conv}, if $\delta_0$ is small enough we have that for every $t\in[0,T(E_0))$ the norm $\Vert\psi_F\Vert_{W^{4,2}(\pa E)}$ is small, hence the same holds for $\Vert\psi_F\Vert_{C^{1}(\pa E)}$. Then, it follows that there exists a constant $C_0=C_0(E,M_E)>0$ such that, for every $i\in\II_E$, we have $\Vert \langle e_i\,|\,\nu_t\rangle\Vert_{L^2(\pa E_t)}\geqslant C_0>0$, holding $\Vert \langle e_i\,|\,\nu_E\rangle\Vert_{L^2(\pa E)}>0$ (notice that this argument also shows that, with an appropriate choice of small $\delta_0$ and $\delta$, the condition $\limsup_{t\to T(E_0)}\|\psi_t\|_{C^{1}(\pa E)}=2M_E$ cannot occur). It is then easy to see that the vector $\eta_t\in\OO_E$ realizing the above minimum for $E_t$ is unique and satisfies
\begin{equation} \label{not a translation 2}
\Delta \HHH = \langle \eta_t\,|\, \nu_t \rangle + f,
\end{equation} 
where $f\in L^2(\pa E_t)$ is a function $L^2$--orthogonal (with respect to the measure $\mu_t$ on $\partial E_t$) to the vector subspace of $L^2(\pa E_t)$ spanned by the functions $\langle e_i\,|\,\nu_t\rangle$. Moreover, letting $\eta_t= \eta^i_t e_i$, from relation~\eqref{not a translation 2} we have
\begin{equation}\label{eqcar50002}
\norma{\Delta \HHH}_{L^2(\pa E_t)}^2 \geqslant \norma{ \langle \eta_t\,|\,\nu_t\rangle}_{L^2(\pa E_t)}^2 = \int_{\pa E_t} \vert\eta^i_t \langle e_i\,|\,\nu_t\rangle \vert^2 \, \dmu_t \geqslant C_0^2 |\eta^i_t|^2 = C\vert \eta_t \vert^2 \, ,
\end{equation}
where $C$ is a constant depending only on $E$ and $M_E$.\\
We now argue by contradiction, assuming $\|f\|_{L^2(\pa E_t)} < \theta \norma{\Delta \HHH}_{L^2(\pa E_t)}$.\\
Integrating by parts and using the Cauchy--Schwarz inequality, we have
\begin{align}
\norma{\grad \HHH}^2_{L^2(\pa E_t)}&= \int_{\pa E_t}\abs{\grad \HHH}^2\, \dmu_t 
=- \int_{\pa E_t}\HHH \Delta \HHH \, \dmu_t =-\int_{\pa E_t}(\HHH- \overline \HHH) \Delta \HHH \, \dmu_t \nonumber
\\&\leqslant \norma{\HHH-\overline \HHH}_{L^2(\pa E_t)}\norma{\Delta \HHH}_{L^2(\pa E_t)} \,, \label{de02}
\end{align}
hence, thanks to inequality~\eqref{Poincareunif}, it follows
\begin{equation}\label{de02bis}
\norma{\grad \HHH}_{L^2(\pa E_t)} \leq C \norma{\Delta \HHH}_{L^2(\pa E_t)}\, .
\end{equation}

Thus, by multiplying relation~\eqref{not a translation 2} by $\HHH-\overline \HHH$ and integrating over $\pa E_t$, we get
\begin{align}
\biggl \rvert \int_{\partial E_t} (\HHH -\overline \HHH)\Delta \HHH \, \dmu_t \biggl \lvert &= 
\biggl \rvert \int_{\partial E_t} (\HHH -\overline \HHH)f \, \dmu_t \biggl \rvert \\
&< \theta \norma{ \HHH -\overline \HHH}_{L^2(\partial E_t)} \norma{ \Delta \HHH }_{L^2(\pa E_t)}\\
&\leqslant C\theta\norma{ \Delta \HHH }_{L^2(\pa E_t)}^2 \, ,\label{de04}
\end{align}
where the equality follows from the identities
\beq\label{ident}
\int_{\pa E_t}\HHH \, \nu_t\, \dmu_t=0\qquad\text{ and }\qquad \int_{\pa E_t}\nu_t\, \dmu_t=0
\eeq
holding for every embedded hypersurface. We notice that the first identity is a consequence of relation~\eqref{lap} and the divergence theorem, while the second is equality~\eqref{divnorm}. Then, recalling equality~\eqref{not a translation 2}, estimate~\eqref{eqcar50002} and the fact that $f$ is $L^2$--orthogonal to $\langle \eta_t\,|\, \nu_t \rangle$, we have
$$
\begin{aligned}
\norma{ \langle \eta_t\,|\, \nu_t \rangle }^2_{L^2(\pa E_t)}&=\int_{\pa E_t}\Delta \HHH \langle \eta_t\,|\,\nu_t \rangle \, \dmu_t\\
&= - \int_{\partial E_t} g^{ij} \nabla_i \HHH\nabla_j \langle \eta_t\,|\, \nu_t\rangle \rangle \, \dmu_t\\
&\leqslant\abs{\eta_t}\norma{\nabla\nu_t}_{L^2(\pa E_t)} \norma{ \nabla \HHH }_{L^2(\pa E_t)}\\
&\leqslant C \norma{ \Delta \HHH}_{L^2(\pa E_t)} \norma{\nabla\nu_t}_{L^2(\pa E_t)}\biggl\vert\int_{\partial E_t} (\HHH -\overline \HHH)\Delta \HHH \, \dmu_t\, \biggr\vert^{1/2}\\
&\leqslant C \sqrt{\theta}\norma{\Delta \HHH}_{L^2(\pa E_t)}^2 \,,
\end{aligned}
$$
where in the last inequality we used equality~\eqref{de04} and we estimated $\norma{\nabla\nu_t}_{L^2(\pa E_t)}$ by inequality~\eqref{eqcar50005} and the fact that $\mathcal{F}(t)\leqslant \delta_0$, as $\nabla\nu_t=\BBB$ by the Gauss--Weingarten relations~\eqref{GW}.\\
If then $\theta>0$ is chosen so small that $C\sqrt{\theta} <1-\theta^2$ in the last inequality, we have a contradiction since equality~\eqref{not a translation 2} and the fact that $\|f\|_{L^2(\pa E_t)}<\theta \norma{\Delta \HHH}_{L^2(\pa E_t)}$ imply (by $L^2$--orthogonality) that
$$
\|\langle \eta_t\,|\, \nu_t\rangle\|^2_{L^2(\pa E_t)}>(1-\theta^2)\norma{\Delta \HHH}_{L^2(\pa E_t)}^2\,.
$$
All this argument shows that with such a choice of $\theta$, condition~\eqref{not a translation} holds, hence by Lemma~\ref{stimaPi} we conclude
\begin{equation}\label{de03}
\Pi_{E_t}(\Delta \HHH)\geqslant\sigma_{\theta} \| \Delta \HHH\|_{L^2(\pa E_t)}^2\qquad \text{ for all $t\in [0, T(E_0))$.}
\end{equation}
Then, estimate~\eqref{de02bis} clearly proves assumption~\eqref{claim} and the proof of monotonicity of $\mathcal E$ in Step~$1$ is concluded. Hence, if $\delta$ is small enough, $\mathcal E(t)$ remains bounded by $\delta$ during the flow, up to the time $t=T(E_0)$, thus the same clearly holds for ${\mathcal F}(t)$.

\medskip

\noindent \textbf{Step ${\mathbf 3}$} (\textit{Global existence of the flow})\textbf{.}\\
We have seen at Step~1 that choosing an appropriate constant $K$, if $\delta$ is small enough, then the ``energy'' $\mathcal E(t)$ is uniformly bounded and decreasing. More precisely, integrating the differential inequality~\eqref{stimadec}, there holds
\begin{equation}\label{final00}
\mathcal E(t)\leqslant \mathcal E(0)e^{-t/K}\leqslant \delta e^{-t/K}\leqslant\delta
\end{equation}
hence, we also have $\mathcal F(t)\leqslant \delta e^{-t/K}\leqslant\delta$, for every $t\in[0,T(E_0))$.\\
Moreover, at Step~2 we already saw that if $\delta_0$ is chosen small enough,
$$
\limsup_{t\to T(E_0)}\|\psi_t\|_{C^{1}(\pa E)}=2M_E
$$
is not possible. Hence, in order to obtain the global existence of the flow, we only have to show that also 
\beq\label{eqcar878}
\limsup_{t\to T(E_0)}\vol(E_t\triangle E)=\delta_0
\eeq
cannot occur.

We define the following quantity
\begin{equation}\label{D(F)0}
D(t)=\int_{E_t\triangle E}d(x, \pa E)\, dx=\int_{E_t} d_E(x)\, dx-\int_E d_E(x)\, dx,
\end{equation}
where $d_E:N_\varepsilon\to\R$ is the signed distance function defined in formula~\eqref{sign dist}. We observe that,
$$
\vol({E_t}\triangle E)\leqslant C\|\psi_t\|_{L^1(\pa E)} \leqslant C\|\psi_t\|_{L^2(\pa E)}
$$
and 
\begin{align}
\|\psi_t\|_{L^2(\pa E)}^2&=2\int_{\pa E} \int_0^{|\psi_t(y)|} t \, dt\,d\mu(y) \\
&=2\int_{\pa E} \int_0^{|\psi_t(y)|} d(L(y,t),\pa E) \, dt\,d\mu(y) \\
&=2\int_{E_t\triangle E} d(x,\pa E)\,JL^{-1}(x)\,dx\\
&\leqslant C D(t)\,,
\end{align}
where the constants depend on $E$ and $M_E$, $L:\partial E\times (-\eps,\eps)\to N_\eps$ is the smooth diffeomorphism defined in formula~\eqref{eqcar410} and $JL$ is its Jacobian. It clearly follows
\begin{equation}\label{D(F)}
\vol({E_t}\triangle E)\leqslant C\|\psi_t\|_{L^1(\pa E)} \leqslant C\|\psi_t\|_{L^2(\pa E)}\leqslant C\sqrt{D(t)}\,,
\end{equation}
and 
\begin{equation}\label{D(F)0bis}
D(t)\leqslant \int_{{E_t}\triangle E}2M_E\, dx=2M_E\vol({E_t}\triangle E)\,.
\end{equation}
Then, recalling formula~\eqref{D(F)0}, we compute
\begin{align}
\frac{d}{dt}D(t)=&\,\frac{d}{dt}\int_{E_t\triangle E} \!\!d(x,\pa E)\, dx=\int_{\pa E_t}\!\!d_E\, \Delta \HHH\, \dmu_t\leqslant C\norma{\Delta\HHH}_{L^2(\pa E_t)}\leqslant C\sqrt{\delta} \,{e}^{-t/2K},
\end{align}
for all $t \leqslant T(E_0)$, where the last inequality clearly follows from the above estimate~\eqref{final00} for ${\mathcal E}(t)$.\\
By integrating this differential inequality on $[0,\overline{t})$ with $\overline{t} \in [0,T(E_0))$ and taking into account estimate~\eqref{D(F)}, we get
\begin{align}
\vol(E_{\overline{t}}\triangle E) & \leqslant C\|\psi_{\overline{t}}\|_{L^2(\pa E)}\leqslant C\sqrt{D(0)+2KC\sqrt{\delta}\,}\leqslant C\sqrt[4]{\delta}\,,\label{step33_2}
\end{align}
as $D(0)\leqslant C\vol(E_0\triangle E)\leqslant C\delta$, by inequality~\eqref{D(F)0bis} with $t=0$. Hence, if $\delta>0$ is small enough such that $C\sqrt[4]{\delta}<\delta_0$, we have that also condition~\eqref{eqcar878} cannot happen.

We conclude that the surface diffusion flow of $E_0$ exists smooth for every time, moreover $E_t\in \mathfrak{C}^{1}_{2M_E}(E)$ and
\begin{equation}\label{final}
\vol(E_{t}\triangle E)\leqslant C\sqrt[4]{\delta}\,,\qquad\qquad\int_{\pa E_t} |\grad^2\HHH|^2\, d\mu_t+ \int_{\pa E_t} |\nabla \HHH|^2\, d\mu_t\leqslant \delta e^{-t/K}\,,
\end{equation}
for every $t\in[0,+\infty)$.

\medskip

\noindent \textbf{Step ${\mathbf 4}$} (\textit{Convergence, up to a subsequence, to a translate of $E$})\textbf{.}\\
Let $t_i\to +\infty$, then by estimates~\eqref{final}, the sets $E_{t_i}$ satisfy the hypotheses of the last point of Lemma~\ref{w42conv}, hence, up to a (not relabeled) subsequence, we have that there exists a critical set $E'\in \mathfrak{C}^{1}_{2M_E}(E)$ such that $E_{t_i}\to E'$ in $W^{3,p}$, that is $\|\psi_{t_i}-\psi\|_{W^{3,p}(\pa E)}\to0$ for some $\psi:\pa E\to\R$ representing $\pa E'$ as a ``normal graph'' on $\pa E$. As $\pa E'$ has constant mean curvature and it is a graph over $\pa E$ of a $C^{1}$ function (by Sobolev embeddings), it follows by standard regularity theory for quasilinear equations that it is smooth (see~\cite{GT} for instance), then by Proposition~2.7 in~\cite{AcFuMoJu} (see also~\cite[Proposition~2.36]{DDMsurvey}), we have that $E'=E+\eta$ for some (small) $\eta\in \R^4$. Such proposition actually states that $E$ is a strict local minimum for the volume--constrained Area functional, up to translations and that a smooth set ``close enough'' to $E$ (as $E'$ in our situation) can be a critical set if and only if it is a translate of $E$.

\medskip

\noindent \textbf{Step ${\mathbf 5}$} (\textit{Smooth exponential convergence of the full sequence})\textbf{.}\\
Arguing similarly as above, we consider the function
$$
\overline{D}(t)=\int_{E_t\triangle E'}d(x,\pa E)\, dx
$$
with derivative
\begin{equation}\label{step6-}
\frac{d}{dt}\overline{D}(t)=\frac{d}{dt}\int_{E_t\triangle E'} d(x,\pa E)\, dx\,=\int_{\pa E_t}\mathrm{sgn}(\psi_t-\psi)\,d_{\pa E}\, \Delta \HHH\, \dmu_t\,,
\end{equation}
where $\mathrm{sgn}$ is the ``sign function''. By the exponential second estimate~\eqref{final} and the fact that $E_t\in \mathfrak{C}^{1}_{2M_E}(E)$, we have
\begin{equation}\label{step6}
\biggl\vert\frac{d}{dt}\overline{D}(t)\biggl\vert\,\leqslant C\norma{\Delta\HHH}_{L^2(\pa E_t)}\leqslant C\sqrt{\delta} \,{e}^{-t/2K}
\end{equation}
for all $t\geqslant0$, moreover, 
\begin{equation}\label{D(F)0tris}
\overline{D}(t)\leqslant \int_{{E_t}\triangle E'}2M_E\, dx=2M_E\vol({E_t}\triangle E')\leqslant C\|\psi_t-\psi\|_{L^1(\pa E)} \leqslant C\|\psi_t-\psi\|_{L^2(\pa E)}
\end{equation}
which implies $\overline{D}(t_i)\to 0$, as $i\to\infty$, by the previous step.\\
Integrating the differential inequality~\eqref{step6-}, we get
\begin{align}\label{step61-}
\overline{D}(t)-\overline{D}(t_i)=&\,-\int_t^{t_i}\frac{d}{ds} \overline{D}(s)\,ds\leqslant\int_t^{+\infty}\biggl\vert\frac{d}{ds} \overline{D}(s)\,ds\,\biggl\vert\,\leqslant\int_t^{+\infty}C\sqrt{\delta} e^{-s/2K}\, ds\\
\leqslant&\, \,2CK\sqrt{\delta}e^{-t/2K}\, ,
\end{align}
hence, passing to the limit as $i\to\infty$, we conclude
\begin{equation}\label{step61}
\overline{D}(t)\leqslant Ce^{-t/2K}
\end{equation}
for every $t\geqslant 0$, thus $\lim_{t\to+\infty}\overline{D}(t)=0$. Then, we have
\begin{align}
\|\psi_t-\psi\|_{L^2(\pa E)}^2&=2\int_{\pa E} \biggl\vert\int_{0}^{\psi_t(y)-\psi(y)} s\, ds\,\biggl\vert\,d\mu(y) \\
&=2\int_{\pa E} \biggl\vert\int_{0}^{\psi_t(y)-\psi(y)} d(L(y,s),\pa E) \, ds\,\biggl\vert\,d\mu(y) \\
&=2\int_{E_t\triangle E'} d(x,\pa E)\,JL^{-1}(x)\,dx\\
&\leqslant C \overline{D}(t)\\
&\,\leqslant Ce^{-t/2K}\,,
\end{align}
where $L:\partial E\times (-\eps,\eps)\to N_\eps$ is, as before, the smooth diffeomorphism defined in formula~\eqref{eqcar410} with Jacobian $JL$. By this exponential decay and the uniform bound on $\|\psi_t-\psi\|_{W^{4,2}(\pa E)}$ following from estimates~\eqref{final} by means of Lemma~\ref{w42conv}, we obtain the convergence of the full sequence $E_t$ to $E'$ in $W^{3,p}$.\\
Finally, we have that the convergence of $E_t\to E+\eta$ is actually exponentially smooth, by  arguing as in the proof of Theorem~5.1 in~\cite{FusJulMor18} (see also~\cite{DDKK}), that is, via standard parabolic estimates and uniform interpolation inequalities (and Sobolev embeddings), holding the exponential convergence in $W^{3,p}$.
\end{proof}

\section{General dimensions}
In general dimension $n\in\N$, the main global existence and stability Theorem~\ref{existence2} takes the following form.

\begin{thm}\label{existence3}~\cite[Theorem~3.3.14]{DianaPhD}
Let $E\subseteq\T^n$, for $n\geqslant 3$, be a strictly stable critical set for the Area functional under a volume constraint and let $N_\eps$ be a tubular neighborhood of $\pa E$, as in formula~\eqref{tubdef}. For $M_E<\eps/2$ small enough, there exists $\delta>0$ such that, if $E_0$ is a smooth set in $\mathfrak C^{1}_{M_E}(E)$ satisfying $\vol(E_0)= \vol(E)$ and
\beq\label{eqcar777n}
\vol(E_0\triangle E)\leqslant\delta\qquad\text{ and }\qquad\int_{\pa E_0} \vert \nabla^{n-2}\HHH_0\vert^2\, \dmu_0 + \int_{\pa E_0} \vert \nabla \HHH_0\vert^2\, \dmu_0 \leqslant \delta\, ,
\eeq
then, the unique smooth surface diffusion flow $E_t$ starting from $E_0$, given by Proposition~\ref{th:EMS1}, is defined for all $t\geqslant0$. Moreover, $E_t$ converges smoothly to $E'=E+\eta$ exponentially fast as $t\to +\infty$, for some $\eta\in \R^n$, with the meaning that the sequence of smooth functions $\psi_t : \pa E \to \R$ representing $\pa E_t$ as ``normal graphs'' on $\pa E$, that is,
$$
\pa E_t= \{ y+ \psi_{t} (y) \nu_{E}(y) \, : \, y \in \pa E\},
$$
satisfy, for every $k\in\N$,
$$
\Vert \psi_{t}-\psi\Vert_{C^k(\pa E)}\leqslant C_ke^{-\beta_k t}\, ,
$$
for every $t \in [0, +\infty)$, for some positive constants $C_k$ and $\beta_k$, where $\psi: \pa E\to \R$ represents $\pa E'=\pa E+\eta$ as a ``normal graph'' on $\pa E$.
\end{thm}

The case $n=2$, where the boundary hypersurfaces of strictly stable critical sets are circles or straight curves in $\T^2$, was analyzed by Elliott and Garcke in~\cite{EllGar}. In the three--dimensional case, this theorem is Theorem~4.3 in the paper~\cite{AcFuMoJu} by~Acerbi, Fusco, Julin and Morini. We now list the appropriate modification to the line of proof in dimension $n=4$ in order to deal with the general dimensional case, noticing that actually up to Proposition~\ref{computeenvar} all the statements are $n$--dimensional. See the PhD thesis of Antonia Diana~\cite{DianaPhD} for the full details.

\begin{itemize}
\item[\ding{111}] {\em Proposition~\ref{computeenvar} in dimension $n\in\N$}.
\smallskip

\noindent Equation~\eqref{energyfinale} must be substituted by 
\begin{align}
\frac{d}{dt}   \int_{\pa E_t} \abs{\grad^{n-2}\HHH}^2 \, \dmu_t  =&- 2  \int_{\pa E_t} |\nabla^{n}\HHH|^2 \, \dmu_t+\int_{\pa E_t} \qol^{2n+2}(\nabla^{n-4}\BBB,\nabla^{n-1}\HHH)\,\dmu_t
\\&+\int_{\pa E_t} \qol^{2n+2}(\nabla^{n-3}\BBB,\nabla^{n}\HHH)\,\dmu_t\label{nequz}
\end{align}
where
\begin{itemize}
\item[$\bullet$] every ``monomial'' of $\qol^{2n+2}(\grad^{n-4}\BBB,\nabla^{n-1}\HHH)$ has $4$ factors in $\BBB$, $\nabla\HHH$ and their covariant derivatives. The factor $\BBB$ (or $\HHH$ without derivatives) or one of its covariant derivatives up to $\nabla^{n-4}\BBB$ is present exactly one time and the other three factors are derivatives of $\nabla\HHH$ up to $\nabla^{n-1}\HHH$, with $\grad^{n-1} \HHH$ or $\grad^{n-2} \HHH$ present at least one time. Moreover, if the factor $\grad^{n-1} \HHH$ is not present, $\BBB$ cannot appear without derivatives;
\item[$\bullet$] every ``monomial'' of $\qol^{2n+2}(\nabla^{n-3}\BBB,\nabla^{n}\HHH)$ has $4$ factors in $\BBB$, $\grad \HHH$ and their covariant derivatives. The factor $\BBB^2$ or one of its covariant derivatives up to $\nabla^{n-4}(\BBB * \nabla \BBB)$ is present exactly one time, the other two factors are derivatives of $\nabla\HHH$ up to $\grad^n \HHH$. The factor $\grad^{n} \HHH$ is present exactly one time, with the exception of ``monomials'' of kind $\nabla^{n-1}\HHH*\BBB^2*\nabla^{n-1}\HHH$.
\end{itemize}

\medskip

\noindent The proof of this formula can be obtained by following step by step the proof of Proposition~\ref{computeenvar}, with the appropriate modifications due to the dimension $n$. The only term that needs a slightly different treatment is the analogue of the last one in formula~\eqref{conto3}, coming from the term
$$
-\int_{\pa E_t} g\big(\nabla^{n-2}\HHH, \nabla^{n-2}(\Delta\HHH\vert\BBB\vert^2)\big)\,\dmu_t
$$
which would appear in the right hand side of the first line of the $n$--dimensional version of computation~\eqref{conto2}. Integrating by parts two times, the integrand becomes a contraction of $\nabla^{n}\HHH$ with $\nabla^{n-4}(\Delta\HHH\vert\BBB\vert^2)$, which clearly is a ``polynomial'' of the form $ \qol^{2n+2}(\nabla^{n-3}\BBB,\nabla^{n}\HHH)$. We refer to~\cite[Lemma~3.3.9]{DianaPhD} for a complete and detailed proof.

\bigskip

\item[\ding{111}] {\em Proposition~\ref{estimateenvar} in dimension $n\in\N$}.
\smallskip

\noindent Inequality~\eqref{energy1} must be substituted by 

\beq \label{stimaenn}
\frac{d}{dt} \int_{\pa E_t} \abs{\grad^{n-2}\HHH}^2 \, \dmu_t \, \leqslant- \norma{\nabla^{n}\HHH}^2_{L^2(\pa E_t)}+C \norma{\grad \HHH}_{L^2(\pa E_t)}^2
\eeq
for any $\eps>0$, with some constant $C$ depending on $E$, $M_E$, $\eps$, $\Vert \grad^{n-3} \BBB \Vert_{L^{\frac{2n-2}{n-3}}(\pa E)}$ and $\|\BBB\|_{L^\infty(\pa E)}$.

\medskip

\noindent The proof of this inequality goes like in the proof of Proposition~\ref{estimateenvar}, by choosing suitable exponents in the interpolation inequalities in dealing with the two ``polynomial'' terms
$$
\int_{\pa E_t} \qol^{2n+2}(\nabla^{n-4}\BBB,\nabla^{n-1}\HHH)\,\dmu_t\qquad\text{ and }\qquad\int_{\pa E_t} \qol^{2n+2}(\nabla^{n-3}\BBB,\nabla^{n}\HHH)\,\dmu_t
$$
in equality~\eqref{nequz}. In particular, expanding the iterated derivatives of $\BBB^2$ in the second one, one gets factors like $\nabla^i\BBB*\nabla^j\BBB$ with $i+j\leqslant n-3$ that are estimated by the constant $C$, after noticing that we uniformly control with $\Vert \grad^{n-3} \BBB \Vert_{L^{\frac{2n-2}{n-3}}(\pa F)}$ and $\|\BBB\|_{L^\infty(\pa F)}$ all the ``intermediate'' norms. We refer to~\cite[Proposition~3.3.11]{DianaPhD} for a complete and detailed proof.

\bigskip

\item[\ding{111}] {\em Lemma~\ref{w42conv} in dimension $n\in\N$}.
\begin{lem}\label{wn2conv}~\cite[Lemma~3.3.13]{DianaPhD}
Let $E\subseteq\T^n$ be a smooth set and $N_\eps$ be a tubular neighborhood of $\pa E$. For $M_E$ small enough and $\delta>0$, there exists a constant $C=C(E,M_E,\delta)$ such that if $F\in\mathfrak{C}^{1}_{M_E}(E)$ with
\begin{equation}
\pa F= \{y + \psi_F (y) \nu_E(y) \, : \, y \in \pa E\}
\end{equation}
for a smooth function $\psi_F:\partial E\to\R$ and
\beq\label{ex-de02-bis}
\int_{\pa F} |\nabla^{n-2}\HHH|^2\,\dmu + \int_{\pa F} |\nabla \HHH|^2\,\dmu \leqslant \delta\,,
\eeq
there hold
\begin{equation}
\|\BBB\|_{L^\infty(\pa F)}+\|\nabla^{n-3}\BBB\|_{L^{\frac{2n-2}{n-3}}(\pa F)}\leqslant C\qquad\text{ and }\qquad
\|\psi_F\|_{W^{n,2}(\pa E)}\leqslant C\,.\label{eqcar50001-bis}
\end{equation}
Moreover, for every $1\leqslant p<\frac{2n-2}{n-3}$, there exists a monotone non--decreasing function $\omega:\R^+\to\R^+$, depending only on $E$ and $M_E$, with $\lim_{\delta\to0^+}\omega(\delta)=0$ and such that if $F$ satisfies the further condition
$$
\vol(F\triangle E)\leqslant \delta\,,
$$
then $\Vert\psi_F\Vert_{W^{n-1,p}(\pa E)}\leqslant\omega(\delta)$.\\
As a consequence, if $E_i\subseteq \mathfrak{C}^{1}_{M_E}(E) $ is a sequence of smooth sets such that 
$$
\sup_{i\in\N}\,\int_{\pa E_i}\abs{\nabla^2\HHH}^{n-2}\, d\mu_i+\int_{\pa E_i}\abs{\nabla\HHH}^2\, \dmu_i<+\infty\,,
$$
then there exists a (non necessarily smooth) set $E'\in \mathfrak{C}^{1}_{M_E}(E)$ such that, up to a (non relabeled) subsequence, $E_i\to E'$ in $W^{n-1,p}$ as $i\to\infty$, for all $1\leqslant p<\frac{2n-2}{n-3}$. Moreover, if
$$
\int_{\pa E_i}|\nabla^{n-2}\HHH|^2\, d\mu_i+\int_{\pa E_i}|\nabla\HHH|^2\, d\mu_i \to 0\,,
$$
as $i\to\infty$, the set $E'$ is critical for the volume--constrained Area functional, that is, its mean curvature is constant.
\end{lem}
\end{itemize}

\begin{proof}[Proof of Theorem~\ref{existence3}] As in the proof of Theorem~\ref{existence2}, by choosing $M_E$ small enough in order that all the constants in the inequalities are uniform and after choosing some small $\delta_0$, we consider the surface diffusion flow $E_t$ starting from $E_0 \in \mathfrak C^{1}_{M_E}(E)$ as in the hypotheses and the maximal time $T(E_0)$ such that $E_t \in C^{1}_{2M_E}(E)$,
$$
\vol(E_t\triangle E)\leqslant\delta_0\qquad\text{ and }\qquad\mathcal F_n(t) =\int_{\pa E_t} |\grad^{n-2}\HHH|^2\, d\mu_t+ \int_{\pa E_t} |\nabla \HHH|^2\, d\mu_t\leqslant\delta_0\,,
$$
for every $t \in [0, T(E_0))$.\\
Then, we aim to show that the smooth function $\psi_t$, which represents the moving boundary $\pa E_t$ as a normal graph on $\pa E$, does not satisfy the following conditions:
\begin{itemize}
\item $\limsup_{t\to T(E_0)}\|\psi_t\|_{C^{1}(\pa E)}=2M_E$
\item $\limsup_{t\to T(E_0)}{\mathcal F_n}(t)=\delta_0$
\item $\limsup_{t\to T(E_0)}\vol(E_t\triangle E)=\delta_0$
\end{itemize}
that is, the maximal time $T(E_0)= + \infty$, hence the flow exists for all times.\\
So, we define for a suitable $K>2$ the following ``energy'' function
\beq\label{E(t)n}
\E_n(t) = \int_{\pa E_t} |\nabla^{n-2}\HHH|^2\, \dmu_t+ K \int_{\pa E_t} |\nabla \HHH|^2\, \dmu_t\geqslant\mathcal F_n(t)\, .
\eeq
Recalling Proposition~\ref{estimateenvar} and estimate~\eqref{stimaenn} and noticing that Lemma~\ref{stimaPi} holds in any dimension (as it is shown in~\cite{DDMsurvey}), following Step~1 and Step~2 in the proof of Theorem~\ref{existence2}, we have
\beq
\frac{d}{dt} \E_n (t) \leqslant - \E_n(t)/K
\eeq
for every $t \in [0, T(E_0))$, that is, the energy $\E_n$ (hence $\mathcal F_n$) is uniformly bounded from above and decreasing.
Then, the rest of the proof proceeds as the one of Theorem~\ref{existence2}.
\end{proof}

\begin{remark}
The second assumption~\eqref{eqcar777n} in Theorem~\ref{existence3} is both natural, by analogy with the case $n=4$ and technically convenient, since it provides the correct estimate on $\BBB$ in Lemma~\ref{wn2conv} and yields the appropriate interpolation inequalities required in our analysis. However, the choice of $\grad^{n-2} \HHH$ might not be optimal. Assuming an estimate on a lower order derivative could in principle lead to the same conclusion of asymptotic stability.
\end{remark}

\bigskip

\noindent{\large{\textbf{Data availability statement:}}} This article has no associated data.

\bigskip

\noindent{\large{\textbf{Conflict of interest statement:}}} The authors have no financial or proprietary interests in any material discussed in this article.

\bibliographystyle{amsplain}
\bibliography{SDF}

\end{document}